\documentclass[12pt,reqno,a4paper]{amsart}

\usepackage{mlmodern} 
\usepackage[T1]{fontenc}

\usepackage{microtype}
\DisableLigatures[f]{encoding = *, family = * }

\usepackage{amsxtra,mathrsfs,mathtools}
\usepackage[colorlinks=true, pdfstartview=FitW]{hyperref}

\usepackage[capitalise]{cleveref}

\AtBeginDocument{\frenchspacing}

\usepackage{enumitem}
\setlist[enumerate,1]{font=\upshape}

\setlist{nosep}

\allowdisplaybreaks[4]

\newtheorem{theorem}{Theorem}[section]
\newtheorem{proposition}[theorem]{Proposition}
\newtheorem{lemma}[theorem]{Lemma}

\newtheorem{remark}[theorem]{Remark}

\theoremstyle{definition}
\newtheorem{definition}[theorem]{Definition}

\numberwithin{equation}{section}

\title{Variational principles for BS  dimension  under amenable group actions}

\author{Zhongxuan Yang$^{}$}
\address{College of Mathematics and Statistics, Chongqing University, Chongqing 401331, China}
\email{yzx@stu.cqu.edu.cn}



\begin{document}
	\begin{abstract}
		In this manuscript, we focus on the investigation of  the BS dimension and BS packing dimension under amenable group actions. Firstly, we  obtain a Bowen’s equation which illustrate the relation of  BS packing dimension to the  packing topological pressure under amenable group actions. Moreover, we establish the variational principle and inverse variational principle for BS dimension and BS packing dimension under amenable group actions.  Finally, we also get  an analogue of Billingsley’s type theorem for  BS packing dimension under amenable group actions.
		
	\end{abstract}
	
	\maketitle

	\section{Introduction}
	
	Kolmogorov \cite{kolmogorov1958new} introduced measure-theoretical entropy for measure preserving dynamical systems. Later,
	the topological entropy was proposed by Adler, Konheim and McAndrew \cite{adler1965topological} to describe the topological complexity of dynamical systems.  Bowen \cite{bowen1971entropy} gave equivalent notions for topological entropy  via spanning sets and separated sets. Besides, Bowen \cite{bowen1973topological} gave a dimensional description of entropy, which was further investigated by Pesin and Pitskel \cite{pesin1984topological},
	these new perspectives led to fruitful results in dimension theory, ergodic theory, multifractal analysis and other fields of dynamical systems. The basic relation between topological entropy and measure-theoretic entropy is the variational principle which was proved by Goodwyn, Dinaburg and Goodman \cite{goodman1971relating,dinaburg1971connection,goodwyn1969topological}.  Feng et al. \cite{feng2012variational} proposed another dimension version entropy named
	the packing topological entropy, and they also established the variational principles for the Bowen topological entropy and the packing topological entropy. Wang \cite{wang2021some} proved two new variational principles for Bowen and packing topological entropies by introducing Bowen entropy and packing entropy of measures in the sense of Katok.
	 Zheng et al. \cite{zheng2016bowen} extended the Bowen entropy variational principle to amenable group actions. Huang  et al. \cite{huang2019bowen} extended the definition of Bowen topological entropy of subsets to continuous action of amenable groups on a compact metrizable space, and they proved the Bowen topological entropy of subsets for amenable group actions can be determined via the local entropies of measures. Dou et al. \cite{dou2023packing} extended the results of the packing topological entropy  in \cite{feng2012variational} to the packing topological entropy under amenable group actions.

	Topological pressure, as a generalization of topological entropy, was introduced by Ruelle \cite{ruelie1973statistical} and extended to compact spaces with continuous transformations by Walters \cite{walters1982introduction}.  Later, inspired by the dimension theory \cite{pesin2008dimension},  Pesin and Piskel \cite{pesin1984topological} extended the notion of topological pressure on arbitrary subsets for general compact metric space.  Tang et al. \cite{tang2015variational}  generalized the Bowen entropy variational principle in \cite{feng2012variational}  to Pensin-Pitskel pressure. Following the ideas given by Feng and Huang \cite{feng2012variational} and  Wang \cite{wang2021some}, Zhong et al.
	further studied the variational principles of Pesin–Pitskel and packing topological pressures \cite{zhong2023variational}. Recently, Huang et al. \cite{huang2020variational} established the variational principle of Pesin-Pitskel pressure under amenable group actions.  Ding et al. \cite{ding2025packing} generalized the results of \cite{dou2023packing} to the packing topological pressure under amenable group actions.

	In the setting of quasi-circles, Bowen \cite{bowen1979hausdorff} firstly found the Hausdorff dimension of certain compact set is exactly the unique root of the equation defined by the topological pressure of geometric potential function, which was later known as Bowen’s equation. After that, Bowen’s equation has been widely investigated. Barreira and Schmeling \cite{barreira2000sets}  got Bowen’s equation of BS dimension, which was proved to be the unique root of the equation defined by topological pressure of additive potential function. In 2012, Wang et al. \cite{wang2012variational} studied Bowen’s equation of BS packing dimension, which showed BS packing dimension is the unique root of the equation defined by packing topological pressure function. Motivated by the works of  \cite{barreira2000sets,bowen1979hausdorff,jaerisch2014induced}, Xing et al. \cite{xing2015induced}  proved that an important connection between   the classical topological pressure and the induced topological pressure  is Bowen’s equation, they also pointed out  that BS dimension is a special case of the induced topological pressure. Recently, 
	Yang et al. \cite{yang2022bowen} introduced the notion of BS metric mean dimension on arbitrary subsets, then established establish Bowen’s equations for Bowen upper metric mean dimension and variational principles for BS metric mean dimension.
	Ding et al. \cite{ding2023variational} studied the  BS dimension of subsets of finitely generated free semigroup actions and they proved BS dimension of subsets of finitely generated free semigroup actions is the unique root of the equation defined by Pesin–Pitskel topological pressure in \cite{xu2018variational}. Besiders, they also obtained a variational principle for BS dimension. Liu and Peng \cite{liu2025variational} introduced and studied the BS (Barreira-Schmeling) dimension for amenable group actions. They proved a Billingsley type theorem characterizing the BS dimension in terms of measure-theoretic quantities, and established a variational principle that connects the topological and measure-theoretic versions of the BS dimension. Yang \cite{yang2025variational} developed the BS-dimension of subsets using neutralized Bowen open balls instead of usual Bowen balls. He proved  a variational principle between neutralized BS-dimension and measure-theoretic BS-dimension.

	Let $(X, G)$ be a $G$-action topological dynamical system, where $X$ is a compact metric space with metric $d$ and $G$ is a topological group acting continuously on $X$. Throughout this manuscript we assume that $G$ is a countable infinite discrete amenable group. Denote $\mathbb{F}(G)$ the collection of all finite subsets of $ G$.
	Recall that a
	countable discrete group $G$ is amenable if there is a sequence of non-empty finite subsets
	$\{F_n\}_{n=1}^{\infty}$ of $G$ such that 
	$$\lim_{n \to \infty} \frac{|F_n \bigtriangleup gF_n|}{|F_n|}=0, \forall g\in G,$$
	and such $\{F_n\}_{n=1}^{\infty}$ is called as a Følner   sequence in $G$. Since $G$ is infinite, the sequence $|F_n|$ tends to infinity. Without loss of generality we can assume that $|F_n|$ increases when $n$ increases.  One can  refer to \cite{ornstein1987entropy,kerr2016ergodic} for some more details about amenable group actions.

	Let $\{F_n\}_{n=1}^{\infty}$ be a Følner  sequence in $G$. Throughout this manuscript, we always let the Følner  sequence $\{F_n\}_{n=1}^{\infty}$ satisfying $|F_n| \to \infty(n \to \infty)$  be fixed unless otherwise specified. These researches in \cite{barreira2000sets,dou2023packing,wang2012variational,zhong2023variational,wang2021some} provide us with the motivation to explore the variational principles of BS dimension and BS packing dimension under amenable group actions. To see it, in this manuscript, we introduce the notions of BS dimension, BS packing dimension under amenable group actions and several corresponding  measure-theoretic BS dimension. Subsequently, we obtain several variational principles and inverse variational principles, and we list them as follows.

	\begin{theorem}\label{m1}
		Let  $\Phi: X\rightarrow \mathbb{R}$ be a  positive continuous function and  $H$ be a non-empty compact subset of $X$.  Then 
		\begin{align*}
			dim_H^{\widetilde{BS}}(\{F_n\}_{n=1}^{\infty},\Phi)&=\sup\{\underline{P}_{\mu}^{\widetilde{BK}}(\{F_n\}_{n=1}^{\infty},\Phi):\mu\in {M}(X),\mu(H)=1\}\\
			&=\sup\{dim_{\mu}^{\widetilde{BS}}(\{F_n\}_{n=1}^{\infty},\Phi):\mu\in {M}(X),\mu(H)=1\}\\
			&=\sup\{dim_{\mu}^{\widetilde{K}}(\{F_n\}_{n=1}^{\infty},\Phi):\mu\in {M}(X),\mu(H)=1\},
		\end{align*}
		where $dim_H^{\widetilde{BS}}(\{F_n\}_{n=1}^{\infty},\Phi)$, $\underline{P}_{\mu}^{\widetilde{BK}}(\{F_n\}_{n=1}^{\infty},\Phi)$, $dim_{\mu}^{\widetilde{BS}}(\{F_n\}_{n=1}^{\infty},\Phi)$ and $dim_{\mu}^{\widetilde{K}}(\{F_n\}_{n=1}^{\infty},\Phi)$ denote BS dimension of $H$, lower measure-theoretic BS dimension of $\mu$, BS dimension of $\mu$ and BS dimension in the sense of Katok of $\mu$, respectively (more details about these notions in section \ref{Preliminary}).
	\end{theorem}

	\begin{theorem}\label{m2}
		Let  $\Phi: X\rightarrow \mathbb{R}$ be a  positive continuous function and  $\mu\in {M}(X)$. Then
		$$
		\begin{aligned}
			dim_{\mu}^{\widetilde{BS}}(\{F_n\}_{n=1}^{\infty},\Phi)&=dim_{\mu}^{\widetilde{K}}(\{F_n\}_{n=1}^{\infty},\Phi)\\
			&=  \inf \left\{dim_H^{\widetilde{BS}}(\{F_n\}_{n=1}^{\infty},\Phi):  \mu(H)=1\right\}.
		\end{aligned}
		$$
	\end{theorem}
	
	\begin{theorem}\label{m3}
		Let $\{F_n\}_{n=1}^{\infty}$ be a Følner  sequence in G satisfying $\lim_{n \to \infty}\frac{|F_n|}{\log n}=\infty.$ Then for any non-empty analytic subset H of X and a positive continuous function  $\Phi: X \rightarrow \mathbb{R} $.
		Then 
		\begin{align*}
			dim_H^{\widetilde{BSP}}(\{F_n\}_{n=1}^{\infty},\Phi)&=\sup\{\overline{P}_{\mu}^{\widetilde{BK}}(\{F_n\}_{n=1}^{\infty},\Phi):\mu \in \mathcal{M} (X),\mu(H)=1\}\\
			&=\sup\{dim_{\mu}^{\widetilde{BSP}}(\{F_n\}_{n=1}^{\infty},\Phi,\epsilon):\mu \in \mathcal{M} (X),\mu(H)=1\}\\
			&=\sup\{dim_{\mu}^{\widetilde{KP}}(\{F_n\}_{n=1}^{\infty},\Phi):\mu \in \mathcal{M} (X),\mu(H)=1\},
		\end{align*}
		where $dim_H^{\widetilde{BSP}}(\{F_n\}_{n=1}^{\infty},\Phi)$, $\overline{P}_{\mu}^{\widetilde{BK}}(\{F_n\}_{n=1}^{\infty},\Phi)$, $dim_{\mu}^{\widetilde{BSP}}(\{F_n\}_{n=1}^{\infty},\Phi)$ and $dim_{\mu}^{\widetilde{KP}}(\{F_n\}_{n=1}^{\infty},\Phi)$ denote  BS  packing dimension of $H$, upper measure-theoretic BS dimension of $\mu$, BS  packing dimension of $\mu$ and BS packing dimension in the sense of Katok of $\mu$, respectively (more details about these notions in section \ref{packing dimension}).
	\end{theorem}
	
	\begin{theorem}\label{m4}
		Let  $\Phi: X\rightarrow \mathbb{R}$ be a  positive continuous function and  $\mu\in {M}(X)$. Then
		$$
		\begin{aligned}
			dim_{\mu}^{\widetilde{BSP}}(\{F_n\}_{n=1}^{\infty},\Phi)&=dim_{\mu}^{\widetilde{KP}}(\{F_n\}_{n=1}^{\infty},\Phi)\\
			&=  \inf \left\{dim_H^{\widetilde{BSP}}(\{F_n\}_{n=1}^{\infty},\Phi):  \mu(H)=1\right\}.
		\end{aligned}
		$$
	\end{theorem}
	
	This manuscript is organized as follows. In Section 2, we introduce the notions of measure-theoretic BS dimension, BS dimension  and BS dimension in the sense of Katok under amenable group actions. The proofs of main results related to BS packing dimension and related properties are given in Section 3 and Section 4. In Section 5 and 6, we turn to investigate the BS packing dimension, and we prove the variational principle, inverse variational principle and Billingsley’s type theorem for BS packing dimension under amenable group actions.

	\section{BS dimension under amenable group actions}\label{Preliminary}
	In this section, we introduce some notions to appear in this manuscript.
	Let  $\mathcal{C}(X, \mathbb{R})$  denote the Banach space of all continuous real-valued functions on  $X $ equipped with the supremum norm, then for  $\Phi \in \mathcal{C}(X, \mathbb{R})$.

	Given $F\in \mathbb{F}(G)$, $x,y \in X$, the Bowen metric $d_F$  on $X$ is defined by  $d_F(x,y):=\max_{g \in F} \limits d(gx,gy).$ Then  \emph{Bowen open  ball} of radius $\epsilon$  in the metric $d_F$ around $x$ is   given by 
	$$B_F(x,\epsilon)=\{y\in X: d_F(x,y)<\epsilon\},$$
	and \emph{Bowen closed  ball }is given by
	$$\overline B_F(x,\epsilon)=\{y\in X: d_F(x,y)\leq\epsilon\}.$$
	Given $\Phi\in C(X,\mathbb{R})$, where $C(X,\mathbb{R})$ denotes the set of all continuous functions, define
	\begin{align*}
		\Phi_F(x)&=\sum_{g \in F}\Phi(gx),\\
		\Phi_F(x,\epsilon)&=\sup_{y \in B_F(x,\epsilon)}\Phi_F(y),\\
		\overline \Phi_F(x,\epsilon)&=\sup_{y \in \overline B_F(x,\epsilon)}\Phi_F(y).
	\end{align*}

	\subsection{Pesin-Pitskel topological pressure of subsets}\cite{liu2025variational}
	Let  $H\subset X$ be a non-empty subset, $\epsilon>0$, $\Phi \in \mathcal{C}(X, \mathbb{R})$, $N\in \mathbb{N}$ , $s \in \mathbb{R}$. 
	Put
	$$M^B(N,s,\epsilon,H,\{F_n\}_{n=1}^{\infty},\Phi)=\inf\{\sum_i\limits  e^{-s\lvert F_{n_i} \rvert +\Phi_{F_{n_i}}(x_i)}\},$$
	where the infimum  is taken over all  finite or countable cover $\{B_{F_{n_i}}(x_i,\epsilon)\}_{i\in I}$ of $H$ with $n_i \geq N.$
	Since  $M^B(N,s,\epsilon,H,\{F_n\},\Phi)$ is non-decreasing when $N$ increases, the following limit exists.
	\begin{align*}
		M^B(s,\epsilon,H,\{F_n\}_{n=1}^{\infty},\Phi)&=\lim_{N\to \infty}M^B(N,s,\epsilon,H,\{F_n\}_{n=1}^{\infty},\Phi),\\
		P^B(\epsilon,H,\{F_n\}_{n=1}^{\infty},\Phi)&=\sup{\{s:M^B (s,\epsilon,H,\{F_n\}_{n=1}^{\infty},\Phi)=\infty\}}\\
		&=\inf{\{s:M^B (s,\epsilon,H,\{F_n\}_{n=1}^{\infty},\Phi)=0\}},\\
		P^B_H(\{F_n\}_{n=1}^{\infty},\Phi)&=\lim_{\epsilon \to 0}P^B(\epsilon,H,\{F_n\}_{n=1}^{\infty},\Phi).
	\end{align*}
	$P_H^B(\{F_n\}_{n=1}^{\infty},\Phi)$ is called {\it Pesin-Pitskel topological pressure} of the set  $H$ along  $\{F_n\}_{n=1}^{\infty}$ with respect to $\Phi.$

	\subsection{Packing topological pressure of subsets}\cite{ding2025packing}
	Let $H\subset X$ be a non-empty subset, $\epsilon>0$, $\Phi \in \mathcal{C}(X, \mathbb{R})$, $N\in \mathbb{N}$, $s \in \mathbb{R}$.
	Put
	$$M^P(N,s,\epsilon,H,\{F_n\}_{n=1}^{\infty},\Phi)=\sup\{\sum_i\limits  e^{-s\lvert F_{n_i} \rvert +\Phi_{F_{n_i}}(x_i)}\},$$
	where the supremum  is taken over all  finite or countable disjoint $\{\overline B_{F_{n_i}}(x_i,\epsilon)\}_{i\in I}$ with $n_i \geq N,~x_i \in H.$  Since $M^P(N,s,\epsilon,H,\{F_n\}_{n=1}^{\infty},\Phi)$ is decreasing when $N$ increases,
	the following limit exists.
	Set
	$$M^P(s,\epsilon,H,\{F_n\}_{n=1}^{\infty},\Phi)=\lim_{N\to \infty}M^P(N,s,\epsilon,H,\{F_n\}_{n=1}^{\infty},\Phi).$$
	Put
	\begin{align*}
		M^\mathcal{P} (s,\epsilon,H,\{F_n\}_{n=1}^{\infty},\Phi)&=\inf \{{\sum_{i=1}^{\infty}M^P(s,\epsilon,H_i,\{F_n\}_{n=1}^{\infty},\Phi):H \subset \cup_{i=1}^{\infty} H_i}\},\\
		P^P(\epsilon,H,\{F_n\}_{n=1}^{\infty},\Phi)&=\sup{\{s:M^\mathcal{P} (s,\epsilon,H,\{F_n\}_{n=1}^{\infty},\Phi)=\infty\}}\\
		&=\inf{\{s:M^\mathcal{P} (s,\epsilon,H,\{F_n\}_{n=1}^{\infty},\Phi)=0\}},\\
		P^P_H(\{F_n\}_{n=1}^{\infty},\Phi)&=\lim_{\epsilon \to 0}P^P(\epsilon,H,\{F_n\}_{n=1}^{\infty},\Phi).
	\end{align*}
	
	Since $P^P(\epsilon,H,\{F_n\}_{n=1}^{\infty},\Phi)$ is increasing when $\epsilon$ decreases, the above limit exists. Then we call $P^P_H(\{F_n\}_{n=1}^{\infty},\Phi)$ {\it packing topological pressure} of the set $H$ along $\{F_n\}_{n=1}^\infty$  with respect to $f.$
	When $\Phi=0$,  $P^P_H(\{F_n\}_{n=1}^{\infty},0)$ is packing topological entropy $h_{top}^{P}(H,\{F_n\}_{n=1}^{\infty})$ given by Dou, Zheng and Zhou \cite{dou2023packing}.

	\subsection{BS dimension of subsets}  
	In this subsection, we offer a notion of BS dimension for arbitrary subsets  as follows.
	Let  $H\subset X$ be a non-empty subset and $\Phi: X\rightarrow \mathbb{R}$ be a  positive continuous function. Given $\epsilon>0$, $N\in\mathbb{N}$, $s>0$, put
	$$M^{s}_{N,\epsilon}(H,\{F_n\}_{n=1}^{\infty},\Phi)=\inf\sum_{i\in I} \exp{[-s\sup_{y\in B_{F_{n_i}}(x_i,\epsilon)}\Phi_{F_{n_i}}(y)]},$$
	where the infimum is taken over all finite or countable covers $\{B_{F_{n_i}}(x_i,\epsilon)\}_{i\in I}$ of $H$ with $n_i\geq N$ and $x_i\in X$. 
	
	Then the limit $M_{\epsilon}^{s}(H,\{F_n\}_{n=1}^{\infty},\Phi)=\lim_{N\rightarrow \infty} M_{N,\epsilon}^{s}(H,\{F_n\}_{n=1}^{\infty},\Phi)$ exists since $M^{s}_{N,\epsilon}(H,\{F_n\}_{n=1}^{\infty},\Phi)$ is non-decreasing when $N$ increases. The quantity $M^s_{\epsilon}(H,\{F_n\}_{n=1}^{\infty},\Phi)$ has a critical value of parameter $s$ jumping from $\infty$ to $0$. The critical value is defined by 
	$$M_{\epsilon}(H,\{F_n\}_{n=1}^{\infty},\Phi)= \inf\{s:M^s_{\epsilon}(H,\{F_n\}_{n=1}^{\infty},\Phi)=0\}=\sup\{s:M^s_{\epsilon}(H,\{F_n\}_{n=1}^{\infty},\Phi)=\infty\}.$$

	\begin{definition}\label{bs}\cite[Definition 2.2]{liu2025variational}
		The {\it BS dimension  on the set $H$} along $\{F_n\}_{n=1}^\infty$ is defined by 
		$$dim_H^{\widetilde{BS}}(\{F_n\}_{n=1}^{\infty},\Phi)=\lim_{\epsilon\rightarrow 0} M_{\epsilon}(H,\{F_n\}_{n=1}^{\infty},\Phi).$$
		When $\Phi=1$, then $dim_H^{\widetilde{BS}}(\{F_n\}_{n=1}^\infty,1)$ is just the {\it Bowen’s topological entropy on the set $H$} in \cite{zheng2016bowen}.
	\end{definition}
	
	Notice that the  Definition \ref{bs} can be given in an alternative way. Given $B_{F_{n_i}}(x_i,\epsilon)$, we can replace $\sup_{y\in B_{F_{n_i}}(x_i,\epsilon)}\Phi_{n_i}(y)$ by $\Phi_{n_i}(x_i)$ in Definition \ref{bs} to give a new definition. We denote by $\widetilde{M}^{s}_{N,\epsilon}(H,\{F_n\}_{n=1}^{\infty},\Phi)$, $\widetilde{M}^{s}_{\epsilon}(H,\{F_n\}_{n=1}^{\infty},\Phi)$,  $\widetilde{M}_{\epsilon}(H,\{F_n\}_{n=1}^{\infty},\Phi)$ and $\widetilde{dim}_H^{\widetilde{BS}}(T,\{F_n\}_{n=1}^{\infty},\Phi)$ the new corresponding quantities of $M^{s}_{N,\epsilon}(H,\{F_n\}_{n=1}^{\infty},\Phi)$, $M^{s}_{\epsilon}(H,\{F_n\}_{n=1}^{\infty},\Phi)$, $M_{\epsilon}(H,\{F_n\}_{n=1}^{\infty},\Phi)$ and $dim_H^{\widetilde{BS}}(T,\{F_n\}_{n=1}^{\infty},\Phi)$, respectively. 
	\begin{proposition}\label{inter}
		Let  $H\subset X$ be a non-empty subset and $\Phi: X\rightarrow \mathbb{R}$ be a  positive continuous function. Then $$dim_H^{\widetilde{BS}}(\{F_n\}_{n=1}^{\infty},\Phi)=\widetilde{dim}_H^{\widetilde{BS}}(\{F_n\}_{n=1}^{\infty},\Phi)=\lim_{\epsilon\rightarrow 0}\widetilde{M}_{\epsilon}(H,\{F_n\}_{n=1}^{\infty},\Phi).$$
	\end{proposition}
	\begin{proof}
		It is clear that  $dim_H^{\widetilde{BS}}(\{F_n\}_{n=1}^{\infty},\Phi)\leq \widetilde{dim}_H^{\widetilde{BS}}(\{F_n\}_{n=1}^{\infty},\Phi).$
		
		Conversely, we shall prove that $dim_H^{\widetilde{BS}}(\{F_n\}_{n=1}^{\infty},\Phi)\geq \widetilde{dim}_H^{\widetilde{BS}}(\{F_n\}_{n=1}^{\infty},\Phi).$
		Given $s\geq 0$, select any finite or countable covers $\{B_{F_{n_i}}(x_i,\epsilon)\}_{i\in I}$ of $H$ with $n_i\geq N$ and $x_i\in X$.  For any $x,y\in X$, let   $\hat{\Phi}=\min_{x\in X}\Phi(x)$
		and $\Phi_{\epsilon}=\sup\{|\Phi(x)-\Phi(y)|: d(x,y)<2\epsilon\}.$ Notice that
		\begin{align}\label{equ2.1}
			\sup_{y\in B_{F_{n}}(x,\epsilon)}\Phi_{F_{n}}(y)&= \sup_{y\in B_{F_{n}}(x,\epsilon)}[\Phi_{F_{n}}(y)-\Phi_{F_{n}}(x)+\Phi_{F_{n}}(x)]\\
			& \leq \sup_{y\in B_{F_{n}}(x,\epsilon)}[|\Phi_{F_{n}}(y)-\Phi_{F_{n}}(x)|+\Phi_{F_{n}}(x)]\label{equ2.2}\\
			&\leq |F_{n}|\Phi_{\epsilon}+\Phi_{F_{n}}(x).\label{equ2.3}
		\end{align}
		Thus, for $s\geq 0$, we derive that 
		\begin{align*}
			\sum_{i\in I} \exp{[-s\Phi_{F_{n_i}}(x_i)]}&\leq \sum_{i\in I} \exp{[-s(\sup_{y\in B_{F_{n_i}}(x_i,\epsilon)}\Phi_{F_{n_i}}(y)-|F_{n_i}|\Phi_{\epsilon})]}\\&
			\leq \sum_{i\in I} \exp{[-s(\sup_{y\in B_{F_{n_i}}(x_i,\epsilon)}\Phi_{F_{n_i}}(y)-\frac{\sup_{y\in B_{F_{n_i}}(x_i,\epsilon)}\Phi_{F_{n_i}}(y)}{\hat{\Phi}}\Phi_{\epsilon}]}\\
			&=\sum_{i\in I} \exp{[-s\sup_{y\in B_{F_{n_i}}(x_i,\epsilon)}\Phi_{F_{n_i}}(y)(1-\frac{\Phi_{\epsilon}}{\hat{\Phi}})]}.  
		\end{align*}
		
		It follows that,
		$$\widetilde{M}^{s}_{N,\epsilon}(H,\{F_n\}_{n=1}^{\infty},\Phi)\leq M^{s(1-\frac{\Phi_{\epsilon}}{\hat{\Phi}})}_{N,\epsilon}(H,\{F_n\}_{n=1}^{\infty},\Phi).$$
		Letting $N\to \infty$, we have 
		$$\widetilde{M}^{s}_{\epsilon}(H,\{F_n\}_{n=1}^{\infty},\Phi)\leq M^{s(1-\frac{\Phi_{\epsilon}}{\hat{\Phi}})}_{\epsilon}(H,\{F_n\}_{n=1}^{\infty},\Phi).$$
		This implies that 
		$$(1-\frac{\Phi_{\epsilon}}{\hat{\Phi}})\widetilde{M}_{\epsilon}(H,\{F_n\}_{n=1}^{\infty},\Phi)\leq {M}_{\epsilon}(H,\{F_n\}_{n=1}^{\infty},\Phi).$$
		Taking $\epsilon\to 0$, one has $dim_H^{\widetilde{BS}}(\{F_n\}_{n=1}^{\infty},\Phi)\geq \widetilde{dim}_H^{\widetilde{BS}}(\{F_n\}_{n=1}^{\infty},\Phi).$

	\end{proof}

	By the theory of Carath{\'e}odory-Pesin structure, the proof of the  following proposition is standard, one can refer to \cite{pesin2008dimension}.
	\begin{proposition} \label{prop 2.21}	
		(1) If $H_1\subset H_2 \subset X$, then $$dim_{H_1}^{\widetilde{BS}}(\{F_n\}_{n=1}^{\infty},\Phi)\leq dim_{H_2}^{\widetilde{BS}}(\{F_n\}_{n=1}^{\infty},\Phi).$$
		
		(2) If $H=\cup_{i\geq 1}H_i$ is a union of  sets $H_i\subset X$, then $$dim_{H}^{\widetilde{BS}}(\{F_n\}_{n=1}^{\infty},\Phi)=\sup_{i\geq 1}dim_{H_i}^{\widetilde{BS}}(\{F_n\}_{n=1}^{\infty},\Phi).$$
		
	\end{proposition}

	Barreira and Schmeling \cite{barreira2000sets}  proved the BS dimension is the unique root of topological pressure function. Later,  Wang and Chen \cite{wang2012variational} extended this result to packing BS dimension.
	Recently, Liu and Peng proved that BS dimension is the unique root of Pesin–Pitskel topological pressure function under amenable group actions \cite{liu2025variational}.

	

	\begin{theorem}[Bowen’s equation] \label{bowen}\cite[Theorem 2.1]{liu2025variational}
		For any  positive continuous function  $\Phi: X \rightarrow \mathbb{R} $, we have  $dim_H^{\widetilde{BS}}(\{F_n\}_{n=1}^{\infty},\Phi)=t $, where  $t $ is the unique root of the equation  $P_{H}^{B}(\{F_n\}_{n=1}^{\infty}, -t \Phi)=0 .$
	\end{theorem}

	\subsection{Local  measure-theoretic BS dimension}
	We denote the set of all Borel probability measures on $X$ by $M(X)$.
	Let $\Phi: X\rightarrow \mathbb{R}$ be a  positive continuous function, $\mu\in M(X)$, $\epsilon>0$ and $x\in X$, define
	$$\underline{P}_{\mu}^{\widetilde{BK}}(\{F_n\}_{n=1}^{\infty},\Phi,x)=\lim_{\epsilon\rightarrow 0}\liminf_{n\rightarrow \infty} \frac{-\log \mu(B_{F_{n}}(x,\epsilon))}{\Phi_{F_{n}}(x)},$$
	and
	$$\underline{P}_{\mu}^{\widetilde{BK}}(\{F_n\}_{n=1}^{\infty},\Phi)=\int\underline{P}_{\mu}^{\widetilde{BK}}(\{F_n\}_{n=1}^{\infty},\Phi,x)d\mu.$$
	We say $\underline{P}_{\mu}^{\widetilde{BK}}(T,\Phi)$ is  the {\it lower  measure-theoretic BS dimension of $\mu$} (see \cite[Definition 2.3]{liu2025variational}).

	Define
	$$\overline{P}_{\mu}^{\widetilde{BK}}(\{F_n\}_{n=1}^{\infty},\Phi,x)=\lim_{\epsilon\rightarrow 0}\limsup_{n\rightarrow \infty} \frac{-\log \mu(B_{F_{n}}(x,\epsilon))}{\Phi_{F_{n}}(x)},$$
	and
	$$\overline{P}_{\mu}^{\widetilde{BK}}(\{F_n\}_{n=1}^{\infty},\Phi)=\int\overline{P}_{\mu}^{\widetilde{BK}}(\{F_n\}_{n=1}^{\infty},\Phi,x)d\mu.$$
	We say $\underline{P}_{\mu}^{\widetilde{BK}}(T,\Phi)$ is  the {\it upper  measure-theoretic BS dimension of $\mu$} (see \cite[Definition 2.3]{liu2025variational}).

	\begin{remark}
		(1) When $\Phi=1$, then $\underline{P}_{\mu}^{\widetilde{BK}}(\{F_n\}_{n=1}^{\infty},1)$ is just the {\it measure-theoretical lower entropy of Borel probability measure $\mu$} in \cite{zheng2016bowen}; $\overline{P}_{\mu}^{\widetilde{BK}}(\{F_n\}_{n=1}^{\infty},1)$ is just the {\it measure-theoretical upper entropy of Borel probability measure $\mu$} given in \cite{dou2023packing}.

		(2) Notice that by  Monotone Convergence Theorem, we have 
		
		\begin{align*}
			\underline{P}_{\mu}^{\widetilde{BK}}(\{F_n\}_{n=1}^{\infty},\Phi)&=\int\underline{P}_{\mu}^{\widetilde{BK}}(\{F_n\}_{n=1}^{\infty},\Phi,x)d\mu\\
			&=\int \lim_{\epsilon\rightarrow 0}\liminf_{n\rightarrow \infty} \frac{-\log \mu(B_{F_{n}}(x,\epsilon))}{\Phi_{F_{n}}(x)} d\mu\\
			&=\lim_{\epsilon\rightarrow 0}\int \liminf_{n\rightarrow \infty} \frac{-\log \mu(B_{F_{n}}(x,\epsilon))}{\Phi_{F_{n}}(x)} d\mu,
		\end{align*}
		and
		\begin{align*}
			\overline{P}_{\mu}^{\widetilde{BK}}(\{F_n\}_{n=1}^{\infty},\Phi)&=\int\overline{P}_{\mu}^{\widetilde{BK}}(\{F_n\}_{n=1}^{\infty},\Phi,x)d\mu\\
			&=\int \lim_{\epsilon\rightarrow 0}\limsup_{n\rightarrow \infty} \frac{-\log \mu(B_{F_{n}}(x,\epsilon))}{\Phi_{F_{n}}(x)} d\mu\\
			&=\lim_{\epsilon\rightarrow 0}\int \limsup_{n\rightarrow \infty} \frac{-\log \mu(B_{F_{n}}(x,\epsilon))}{\Phi_{F_{n}}(x)} d\mu.
		\end{align*}
		Hence, we can define the above notions by an alternative way, put
		
		$$\underline{P}_{\mu}^{\widetilde{BK}}(\{F_n\}_{n=1}^{\infty},\Phi,\epsilon,x)=\liminf_{n\rightarrow \infty} \frac{-\log \mu(B_{F_{n}}(x,\epsilon))}{\Phi_{F_{n}}(x)},$$
		and
		$$\underline{P}_{\mu}^{\widetilde{BK}}(\{F_n\}_{n=1}^{\infty},\Phi,\epsilon)=\int\underline{P}_{\mu}^{\widetilde{BK}}(T,\Phi,\epsilon,x)d\mu.$$
		Then, the {\it lower  measure-theoretic BS dimension of $\mu$} is given by
		$$\underline{P}_{\mu}^{\widetilde{BK}}(\{F_n\}_{n=1}^{\infty},\Phi)=\lim_{\epsilon\rightarrow 0}\underline{P}_{\mu}^{\widetilde{BK}}(\{F_n\}_{n=1}^{\infty},\Phi,\epsilon).$$

		Let
		$$\overline{P}_{\mu}^{\widetilde{BK}}(\{F_n\}_{n=1}^{\infty},\Phi,\epsilon,x)=\limsup_{n\rightarrow \infty} \frac{-\log \mu(B_{F_{n}}(x,\epsilon))}{\Phi_{F_{n}}(x)},$$
		and
		$$\overline{P}_{\mu}^{\widetilde{BK}}(\{F_n\}_{n=1}^{\infty},\Phi,\epsilon)=\int\overline{P}_{\mu}^{\widetilde{BK}}(T,\Phi,\epsilon,x)d\mu.$$
		Then, the {\it upper  measure-theoretic BS dimension of $\mu$} is defined by
		$$\overline{P}_{\mu}^{\widetilde{BK}}(\{F_n\}_{n=1}^{\infty},\Phi)=\lim_{\epsilon\rightarrow 0}\overline{P}_{\mu}^{\widetilde{BK}}(\{F_n\}_{n=1}^{\infty},\Phi,\epsilon).$$
	\end{remark}

	\subsection{BS dimension in the sense of Katok}\label{subsec 2.6}
	Following the idea of Katok \cite{katok1980lyapunov}, Wang \cite{wang2021some}  established a variational principle for Bowen  topological entropy via introducing Bowen entropy of measures in the sense of Katok. Later, Zhong et al. \cite{zhong2023variational} generalized Wang's results to topological pressure. 
	Their works motivate us to  explore the variational principles of BS dimension via introducing the notion of BS dimension in the sense of Katok as follows.

	Given $\epsilon>0$, $s>0$, $N\in\mathbb{N}$, $\mu\in M(X)$, $\delta\in (0,1)$ and $\Phi: X\rightarrow \mathbb{R}$ be a  positive continuous function. Put 
	\begin{equation}\label{eq22}
		\Lambda^s_{N,\epsilon}(\mu,\{F_n\}_{n=1}^{\infty},\Phi,\delta)=\inf\sum_{i\in I}\exp{[- s\Phi_{n_i}(x_i)]},
	\end{equation}
	where the infimum is taken over all finite or countable covers $\{B_{F_{n_i}}(x_i,\epsilon)\}_{i\in I}$ so that $\mu(\bigcup_{i\in I}B_{F_{n_i}}(x_i,\epsilon))>1-\delta$ with $n_i\geq N$ and $x_i\in X$.
	
	Let $\Lambda^s_{\epsilon}(\mu,\{F_n\}_{n=1}^{\infty},\Phi,\delta)=\lim_{N\rightarrow \infty} \Lambda^s_{N,\epsilon}(\mu,\{F_n\}_{n=1}^{\infty},\Phi,\delta)$. There is a critical value of $s$ for $\Lambda^s_{\epsilon}(\mu,\{F_n\}_{n=1}^{\infty},\Phi,\delta)$ jumping from $\infty$ to $0$. Define the critical value as 
	\begin{align*}
		\Lambda_{\epsilon}(\mu,\{F_n\}_{n=1}^{\infty},\Phi,\delta)&=\inf\{s:\Lambda^s_{\epsilon}(\mu,\{F_n\}_{n=1}^{\infty},\Phi,\delta)=0\}\\
		&=\sup\{s:\Lambda^s_{\epsilon}(\mu,\{F_n\}_{n=1}^{\infty},\Phi,\delta)=\infty\}.
	\end{align*}
	Let $$dim_{\mu}^{\widetilde{K}}(\{F_n\}_{n=1}^{\infty},\Phi,\epsilon)=\lim_{\delta\rightarrow 0} \Lambda_{\epsilon}(\mu,\{F_n\}_{n=1}^{\infty},\Phi,\delta).$$ 
	\begin{definition}\label{def 2.6}
		The {\it BS dimension in the sense of Katok of $\mu$ along $\{F_n\}_{n=1}^\infty$} is defined as 
		$$dim_{\mu}^{\widetilde{K}}(\{F_n\}_{n=1}^{\infty},\Phi)=\lim_{\epsilon \rightarrow 0}dim_{\mu}^{\widetilde{K}}(\{F_n\}_{n=1}^{\infty},\Phi,\epsilon).$$
	\end{definition}
	
	\begin{definition}\label{def 2.7}
		Moreover, let
		$$dim_{\mu}^{\widetilde{BS}}(\{F_n\}_{n=1}^{\infty},\Phi,\epsilon)=\lim_{\delta\rightarrow 0} \inf\{\widetilde{M}_{\epsilon}(H,\{F_n\}_{n=1}^{\infty},\Phi):\mu(H)\geq 1-\delta\}.$$
		The {\it BS dimension of $\mu$ along $\{F_n\}_{n=1}^\infty$ }  is given by  
		$$dim_{\mu}^{\widetilde{BS}}(\{F_n\}_{n=1}^{\infty},\Phi)=\lim_{\epsilon\rightarrow 0}dim_{\mu}^{\widetilde{BS}}(\{F_n\}_{n=1}^{\infty},\Phi,\epsilon).$$
	\end{definition}
	
	\begin{remark}\label{rem2.8}
		With similar arguments in Propostion \ref{inter}, we have 
		
		(1)  $$dim_{\mu}^{\widetilde{BS}}(\{F_n\}_{n=1}^{\infty},\Phi)=\lim_{\epsilon\rightarrow 0}\lim_{\delta\rightarrow 0} \inf\{{M}_{\epsilon}(H,\{F_n\}_{n=1}^{\infty},\Phi):\mu(H)\geq 1-\delta\}.$$
		
		(2) if we replace $\Phi_{n_i}(x_i)$ by $\sup_{y\in B_{F_{n_i}}(x_i,\epsilon)}\Phi_{F_{n_i}}(y)$ in equation (\ref{eq22}), we can define new function $\widetilde{\Lambda}$.
		For any $\epsilon > 0$ and $0 <\delta<1$, we denote the  critical value by  $\widetilde{\Lambda}_{\epsilon}(\mu,\{F_n\}_{n=1}^{\infty},\Phi,\delta)$, then we have 
		$$dim_{\mu}^{\widetilde{K}}(\{F_n\}_{n=1}^{\infty},\Phi,\epsilon)=\lim_{\delta\rightarrow 0} \lim_{\delta\rightarrow 0} \widetilde{\Lambda}_{\epsilon}(\mu,\{F_n\}_{n=1}^{\infty},\Phi,\delta).$$ 
		
		(3) In fact, Definition \ref{def 2.6} and Definition \ref{def 2.7} are natural generalizations given in \cite{wang2021some}.
	\end{remark}
	
	The next result reveals  that the {\it BS dimension of $\mu$ along $\{F_n\}_{n=1}^\infty$}  is equal to \it BS dimension in the sense of Katok of $\mu$ along $\{F_n\}_{n=1}^\infty$.
	
	\begin{theorem}\label{equ} Let $\Phi: X\rightarrow \mathbb{R}$ be a  positive continuous function and $\mu\in M(X)$.
		For any $\epsilon>0$, then $dim_{\mu}^{\widetilde{K}}(\{F_n\}_{n=1}^{\infty},\Phi,\epsilon)=dim_{\mu}^{\widetilde{BS}}(\{F_n\}_{n=1}^{\infty},\Phi,\epsilon).$ Moreover, we have $dim_{\mu}^{\widetilde{BS}}(\{F_n\}_{n=1}^{\infty},\Phi)=dim_{\mu}^{\widetilde{K}}(\{F_n\}_{n=1}^{\infty},\Phi).$
	\end{theorem}
	
	\begin{proof}
		First, we  show that  $dim_{\mu}^{\widetilde{K}}(\{F_n\}_{n=1}^{\infty},\Phi,\epsilon) \leq dim_{\mu}^{\widetilde{BS}}(\{F_n\}_{n=1}^{\infty},\Phi,\epsilon)$. For any  $n \in \mathbb{N}, 0<\epsilon,0<\delta<1 $, and $H\subset X$  with  $\mu(H) \geq 1-\delta $, then one has
		$$
		\Lambda_{N,\epsilon}^s(\mu,\{F_n\}_{n=1}^{\infty},\Phi,\delta)\leq \widetilde{M}_{N,\epsilon}^s(H,\{F_n\}_{n=1}^{\infty},\Phi).
		$$
		
		Letting  $N \rightarrow \infty$,   we get  
		$$
		\Lambda_{\epsilon}^s(\mu,\{F_n\}_{n=1}^{\infty},\Phi,\delta)\leq \widetilde{M}_{\epsilon}^s(H,\{F_n\}_{n=1}^{\infty},\Phi).
		$$
		
		This indicates that
		$$
		\Lambda_{\epsilon}(\mu,\{F_n\}_{n=1}^{\infty},\Phi,\delta)\leq \widetilde{M}_{\epsilon}(H,\{F_n\}_{n=1}^{\infty},\Phi).
		$$
		and then 
		$$
		\Lambda_{\epsilon}(\mu,\{F_n\}_{n=1}^{\infty},\Phi,\delta) \leq \inf \left\{\widetilde{M}_{\epsilon}(H,\{F_n\}_{n=1}^{\infty},\Phi): \mu(H) \geq 1-\delta\right\} .
		$$
		Hence, taking $ \delta \rightarrow 0$, we deduce that  $$dim_{\mu}^{\widetilde{K}}(\{F_n\}_{n=1}^{\infty},\Phi,\epsilon) \leq dim_{\mu}^{\widetilde{BS}}(\{F_n\}_{n=1}^{\infty},\Phi,\epsilon).$$
		
		On the other hand, we shall prove  $dim_{\mu}^{\widetilde{K}}(\{F_n\}_{n=1}^{\infty},\Phi,\epsilon) \geq dim_{\mu}^{\widetilde{BS}}(\{F_n\}_{n=1}^{\infty},\Phi,\epsilon)$, set \\ $\zeta=dim_{\mu}^{\widetilde{K}}(\{F_n\}_{n=1}^{\infty},\Phi,\epsilon)$. For any  $s>0 $,   exists  $\delta_{s} $ so that
		$$
		\Lambda_{\epsilon}(\mu,\{F_n\}_{n=1}^{\infty},\Phi,\delta)<\zeta+s, \forall \delta<\delta_{s} .
		$$
		This yields that  $ \Lambda_{\epsilon}^{\zeta+s}(\mu,\{F_n\}_{n=1}^{\infty},\Phi,\delta)=0 $. For any  $N \in \mathbb{N} $, we can find a sequence of  $\delta_{N, m}$  with  $\lim \limits_{m \rightarrow \infty} \delta_{N, m}=0$  and   $\{B_{F_{n_i}}(x_{i}, \epsilon)\}_{i \in I_{N, m}}$  such that  $x_{i} \in X, n_{i} \geq N ,  \mu\left(\cup_{i \in I_{N, m}} B_{F_{n_i}}(x_{i}, \epsilon)\right) \geq 1-\delta_{N, m} $, and
		$$
		\sum_{i \in I_{N, m}} e^{-(\zeta+s)\Phi_{n_i}(x_i)} \leq \frac{1}{2^{m}} .
		$$
		
		Let
		$$
		H_{N}=\bigcup_{m \in \mathbb{N}} \bigcup_{i \in I_{N, m}} B_{F_{n_i}}(x_{i}, \epsilon).
		$$
		Then  $\mu\left(H_{N}\right)=1 $ and
		${M}_{N,\epsilon}^{\zeta+s}(H_{N},\{F_n\}_{n=1}^{\infty},\Phi)\leq1.$
		Let  $H=\cap_{N \in \mathbb{N}} H_{N} $. Thus  $\mu\left(H\right)=1$  and
		$$
		\widetilde{M}_{N,\epsilon}^{\zeta+s}(H,\{F_n\}_{n=1}^{\infty},\Phi) \leq \widetilde{M}_{N,\epsilon}^{\zeta+s}(H_{N},\{F_n\}_{n=1}^{\infty},\Phi) \leq 1, \forall N \in \mathbb{N}.
		$$
		This indicates that
		$$
		\widetilde{M}_{\epsilon}(H,\{F_n\}_{n=1}^{\infty},\Phi) \leq \zeta+s.
		$$
		
		Thus,
		$$
		dim_{\mu}^{\widetilde{BS}}(\{F_n\}_{n=1}^{\infty},\Phi,\epsilon)=\lim_{\delta\rightarrow 0} \inf\{\widetilde{M}_{\epsilon}(H,\{F_n\}_{n=1}^{\infty},\Phi):\mu(H)\geq 1-\delta\} \leq \zeta+s.
		$$
		Since $ s$ is arbitrary, then we have $$dim_{\mu}^{\widetilde{BS}}(\{F_n\}_{n=1}^{\infty},\Phi,\epsilon) \leq \zeta=dim_{\mu}^{\widetilde{K}}(\{F_n\}_{n=1}^{\infty},\Phi,\epsilon) .$$ 
		
		Therefore,  we have 
		$$dim_{\mu}^{\widetilde{BS}}(\{F_n\}_{n=1}^{\infty},\Phi,\epsilon) =dim_{\mu}^{\widetilde{K}}(\{F_n\}_{n=1}^{\infty},\Phi,\epsilon)$$ 
		and $$dim_{\mu}^{\widetilde{BS}}(\{F_n\}_{n=1}^{\infty},\Phi)=dim_{\mu}^{\widetilde{K}}(\{F_n\}_{n=1}^{\infty},\Phi).$$
		
	\end{proof}

	Next, we show that  the BS dimension in the sense of Katok along $\{F_n\}_{n=1}^\infty$ is a upper bound of lower  measure-theoretic BS dimension along $\{F_n\}_{n=1}^\infty$. Moreover, by Theorem \ref{equ}, we also have that BS dimension of Borel probability measure is also a upper bound of lower  measure-theoretic BS dimension.
	\begin{proposition}\label{prop 2.5}
		Let   $\Phi: X\rightarrow \mathbb{R}$ be a  positive continuous function and  $\mu \in {M}(X)$.  Then  $$\underline{P}_{\mu}^{\widetilde{BK}}(\{F_n\}_{n=1}^{\infty},\Phi)\leq dim_{\mu}^{\widetilde{K}}(\{F_n\}_{n=1}^{\infty},\Phi).$$
	\end{proposition}
	\begin{proof}
		Assume that  $\underline{P}_{\mu}^{\widetilde{BK}}(\{F_n\}_{n=1}^{\infty},\Phi,2\epsilon)>0$.  We define 
		$$E_N=\{x\in X:  \mu (B_{F_{n}}(x, 2\epsilon))<\exp(-s\Phi_{F_{n}}(x))~ \text{for all}~n\geq N\}.$$
		Let  $s<\underline{P}_{\mu}^{\widetilde{BK}}(\{F_n\}_{n=1}^{\infty},\Phi,2\epsilon)$.  Then there  exists $N_0$  so that $\mu(E_{N_0})>0$. Fix  $\delta_0=\frac{1}{2}\mu(E_{N_0})>0$. Let $\{B_{F_{n_i}}(x_i,\epsilon)\}_{i\in I}$ be a  finite or countable cover  so  that  $\mu (\cup_{i\in I}B_{F_{n_i}}(x_i, \epsilon))> 1-\delta_0$  with $n_i \geq N_0,  x_i\in X$. Then $\mu(E_{N_0}\cap  \cup_{i\in I}B_{F_{n_i}}(x_i, \epsilon))\geq \frac{1}{2}\mu(E_{N_0})>0$. Denote by $I_1=\{i\in I: E_{N_0}\cap  B_{F_{n_i}}(x_i, \epsilon) \not =\emptyset\}$. For every $i \in I_1$, we choose $y_i \in E_{N_0}\cap  B_{F_{n_i}}(x_i, \epsilon)$ such that $$E_{N_0}\cap  B_{F_{n_i}}(x_i, \epsilon)\subset  B_{F_{n_i}}(y_i, 2\epsilon).$$ Notice that $ B_{F_{n_i}}(x_i, \epsilon)\subset  B_{F_{n_i}}(y_i, 2\epsilon).$
		Hence, by inequalities (\ref{equ2.1})-(\ref{equ2.3}), one has 
			\begin{align*}
			\sup_{y\in B_{F_{n}}(x,\epsilon)}\Phi_{F_{n}}(y)\leq |F_{n}|\Phi_{\epsilon}+\Phi_{F_{n}}(x)\leq \frac{\Phi_{F_{n}}(x)}{\hat{\Phi}}\Phi_{\epsilon}+\Phi_{F_{n}}(x)=(1+\frac{\Phi_{\epsilon}}{\hat{\Phi}})\Phi_{F_{n}}(x).
		\end{align*}
		Therefore,

		we get that 
		$$
		\begin{aligned}
			\sum_{i\in I}\exp{[- s\frac{1}{1+\frac{\Phi_{\epsilon}}{\hat{\Phi}}}\sup_{z\in B_{F_{n_i}}(x_i,\epsilon)}\Phi_{n_i}(z)]}&\geq \sum_{i\in I_1}\exp{[- s\frac{1}{1+\frac{\Phi_{\epsilon}}{\hat{\Phi}}}\sup_{z\in B_{F_{n_i}}(x_i,\epsilon)}\Phi_{F_{n_i}}(z)]}\\
			&\geq \sum_{i\in I_1}\exp{[- s\frac{1}{1+\frac{\Phi_{\epsilon}}{\hat{\Phi}}}\Phi_{F_{n_i}}(y_i)}(1+\frac{\Phi_{\epsilon}}{\hat{\Phi}})]\\
			&=\sum_{i\in I_1}\exp{[- s\Phi_{F_{n_i}}(y_i)}]\\
			&\geq \sum_{i\in I_1}  \mu(B_{F_{n_i}}(y_i, 2\epsilon))\\
			&\geq \frac{\mu(E_{N_0})}{2}>0,
		\end{aligned}
		$$
		Thus,  $\widetilde{\Lambda}_{\epsilon}^{s\frac{1}{1+\frac{\Phi_{\epsilon}}{\hat{\Phi}}}}(\mu, \{F_n\}_{n=1}^{\infty},\Phi, \delta_0)\geq\widetilde{\Lambda}_{N_0,\epsilon}^{s\frac{1}{1+\frac{\Phi_{\epsilon}}{\hat{\Phi}}}}(\mu,\{F_n\}_{n=1}^{\infty},\Phi, \delta_0)>0$ and hence $\widetilde{\Lambda}_{\epsilon}(\mu, \{F_n\}_{n=1}^{\infty},\Phi, \delta_0)\geq s\frac{1}{1+\frac{\Phi_{\epsilon}}{\hat{\Phi}}}.$ Consequently, $ dim_{\mu}^{\widetilde{K}}(\{F_n\}_{n=1}^{\infty},\Phi,\epsilon) \geq s\frac{1}{1+\frac{\Phi_{\epsilon}}{\hat{\Phi}}}$ by Remark \ref{rem2.8}. 
		
		Hence, taking $s \to  \underline{P}_{\mu}^{\widetilde{BK}}(\{F_n\}_{n=1}^{\infty},\Phi,2\epsilon)$ and $\epsilon\to 0$, we then obtain that $\underline{P}_{\mu}^{\widetilde{BK}}(\{F_n\}_{n=1}^{\infty},\Phi)\leq dim_{\mu}^{\widetilde{K}}(\{F_n\}_{n=1}^{\infty},\Phi)$. The proof is completed.
	\end{proof}

	\section{Variational principle and inverse variational principle for BS dimension}\label{proof}
	
	\subsection{ Weighted BS dimension}

	Let  $\Phi: X\rightarrow \mathbb{R}$ be a  positive continuous function and $g:X\rightarrow \mathbb{R}$ be a bounded real-valued function on $X$. Given $s>0,N\in\mathbb{N}$ and $\epsilon>0$. Define
	$$W^s_{N,\epsilon}(g,\{F_n\}_{n=1}^{\infty},\Phi)=\inf\sum_{i\in I}c_i \exp{[-s\sup_{y\in B_{F_{n_i}}(x_i,\epsilon)}\Phi_{F_{n_i}}(y)]},$$
	where the infimum is taken over all finite or countable families $\{(B_{F_{n_i}}(x_i,\epsilon),c_i)\}_{i\in I}$ with $0<c_i<\infty$, $x_i \in X$ and $n_i\geq N$ so that 
	$$\sum_{i\in I}c_i \chi_{B_{n_i}(x_i,\epsilon)}\geq g,$$ 
	where $\chi_A$ denotes the characteristic function of $A$.
	
	Let $H\subset X$ be a non-empty subset. Let $W^s_{N,\epsilon}(H,\{F_n\}_{n=1}^{\infty},\Phi)= W^s_{N,\epsilon}(\chi_H,\{F_n\}_{n=1}^{\infty},\Phi)$. Let $W^s_{\epsilon}(H,\{F_n\}_{n=1}^{\infty},\Phi)=\lim_{N\rightarrow \infty} W^s_{N,\epsilon}(H,\{F_n\}_{n=1}^{\infty},\Phi)$. There is a critical value of $s$ such that the quantity $W^s_{\epsilon}(H,\{F_n\}_{n=1}^{\infty},\Phi)$ jumps from $\infty$ to $0$. Denote this critical value as 
	$$W_{\epsilon}(H,\{F_n\}_{n=1}^{\infty},\Phi)=\inf\{s:W^s_{\epsilon}(H,\{F_n\}_{n=1}^{\infty},\Phi)=0\}=\sup\{s:W^s_{\epsilon}(H,\{F_n\}_{n=1}^{\infty},\Phi)=\infty\}.$$
	\begin{definition}\label{weighted}\cite{liu2025variational}
		The {\it  weighted BS dimension on the set $H$ along $\{F_n\}_{n=1}^{\infty}$} is defined as 
		$$dim_H^{\widetilde{BS,WB}}(\{F_n\}_{n=1}^{\infty},\Phi)=\lim_{\epsilon\rightarrow 0}  W_{\epsilon}(H,\{F_n\}_{n=1}^{\infty},\Phi).$$
	\end{definition}
	Analogous to Definition \ref{bs},  the notion of weighted BS dimension along $\{F_n\}_{n=1}^{\infty}$ can also be given in an alternative way. Given $B_{F_{n_i}}(x_i,\epsilon)$, we can replace $\sup_{y\in B_{F_{n_i}}(x_i,\epsilon)}\Phi_{F_{n_i}}(y)$ by $\Phi_{F_{n_i}}(x_i)$ in Definition \ref{weighted} to give a new definition. We denote by $\widetilde{W}^{s}_{N,\epsilon}(H,\{F_n\}_{n=1}^{\infty},\Phi)$, $\widetilde{W}_{\epsilon}(H,\{F_n\}_{n=1}^{\infty},\Phi)$ and $\widetilde{dim}_H^{\widetilde{BS,WB}}(\{F_n\}_{n=1}^{\infty},\Phi)$ the new corresponding quantities of $W^{s}_{N,\epsilon}(H,\{F_n\}_{n=1}^{\infty},\Phi)$, $W_{\epsilon}(H,\{F_n\}_{n=1}^{\infty},\Phi)$ and $dim_H^{\widetilde{BS,WB}}(\{F_n\}_{n=1}^{\infty},\Phi)$, respectively.

	\begin{proposition}\label{inter-1}
		Let   $H\subset X$ be a non-empty subset and $\Phi: X\rightarrow \mathbb{R}$ be a  positive continuous function. Then $${dim}_H^{\widetilde{BS,WB}}(\{F_n\}_{n=1}^{\infty},\Phi)=\widetilde{dim}_H^{\widetilde{BS,WB}}(\{F_n\}_{n=1}^{\infty},\Phi).$$
	\end{proposition}
	\begin{proof}
		It is similar to Proposition \ref{inter}, here we omit it.
		
	\end{proof}

	The following covering lemma can help us to build the equivalence of the BS dimension and  the weighted BS dimension under amenable group actions.
	\begin{lemma}\cite{wang2021variational}\label{cover}
		Let $(X,d)$ be a compact metric space. Let $r>0$ and $\mathcal{B}=\{B(x_i,r)\}_{i\in I}$ be a family of open balls of $X$. Define 
		\begin{align*}
			I(i)=\{j\in I:B(x_j,r)\cap B(x_i,r)\not=\emptyset\}.
		\end{align*}
		Then there exists a finite subset $J\subset I$ such that for any $i,j\in J$ with $i\not =j$, $I(i)\cap I(j)=\emptyset$ and 
		$$\bigcup_{i\in I}B(x_i,r)\subset \bigcup_{j\in J}B(x_j,5r).$$
	\end{lemma}
	
	The proof of the following result is inspired by \cite{tang2015variational,feng2012variational,wang2012variational}. N It is worth noting that Liu and Peng also proved similar results in \cite[Proposition 4.1]{liu2025variational}.  Here, we provide  a simpler  proof by using Lemma \ref{cover}.
	\begin{proposition}\label{equality-1}
		Let   $H$ be a non-empty subset of $X$, $\Phi: X\rightarrow \mathbb{R}$ be a  positive continuous function and $\{F_n\}_{n=1}^{\infty}$ be a Følner  sequence satisfying $\lim\limits_{n\to \infty}\frac{|F_{n}|}{\log n}=\infty$. Then for  $\theta>0$, $s>0$ there exists $\epsilon_0>0$ such that for $0<\epsilon<\epsilon_0$, there exists $N_0\in\mathbb{N}$, such that for any $N\geq N_0$,
		$$\widetilde{M}^{s+\theta}_{N,5\epsilon}(H,\{F_n\}_{n=1}^{\infty},\Phi)\leq W^s_{N,\epsilon}(H,\{F_n\}_{n=1}^{\infty},\Phi)\leq M^s_{N,\epsilon}(H,\{F_n\}_{n=1}^{\infty},\Phi).$$
		Therefore, we have ${dim}_H^{\widetilde{BS}}(\{F_n\}_{n=1}^{\infty},\Phi)={dim}_H^{\widetilde{BS,WB}}(\{F_n\}_{n=1}^{\infty},\Phi)$. 
	\end{proposition}

	\begin{proof}
		The inequality $W^s_{N,\epsilon}(H,\{F_n\}_{n=1}^{\infty},\Phi)\leq M^s_{N,\epsilon}(H,\{F_n\}_{n=1}^{\infty},\Phi)$ is directly from the definition. Next we show that $\widetilde{M}^{s+\theta}_{N,\frac{\epsilon}{2}}(H,\{F_n\}_{n=1}^{\infty},\Phi)\leq W^{s}_{N,\epsilon}(H,\{F_n\}_{n=1}^{\infty},\Phi)$.
		Let $\hat{\Phi}=\min_{x\in X}\Phi(x)$. Since  $\Phi: X\rightarrow \mathbb{R}$ is a  positive continuous function, then we have 
		$$\lim_{\epsilon \rightarrow 0}\limsup_{n\to \infty}\frac{\Phi_{F_n}(\epsilon)}{|F_n|}=0.$$
		where $\Phi_{F_{n}}(\epsilon)=\sup\{|\Phi_{F_{n}}(x)-\Phi_{F_{n}}(y)|:d_{F_n}(x,y)<2\epsilon\}.$
		Hence, for $\theta>0$, there exists $\epsilon_0>0$ such that for any $0<\epsilon<\epsilon_0$, there is $N_0$ such that for any $N>N_0$, we have 
		$$\Phi_{F_{n}}(\epsilon)<\frac{|F_n|\theta\hat{\Phi}}{4s}.$$
		
		Fix $N>N_0$ large enough such that  $\frac{n^2}{e^{|F_n|\theta\hat{\Phi}/2}}<1$ for all $n\geq N$. Let $t>0$ and $n\geq N$. Let $\{(B_{F_{n_i}}(x_i,\epsilon),c_i)\}_{i\in I}$ with $0<c_i<\infty, x_i\in X, n_i\geq N$ be a finite or countable family satisfying $\sum_{i\in I} c_i\chi_{B_{F_{n_i}}(x_i,\epsilon)}\geq \chi_{H}$. Define $I_n=\{i\in I:n_i=n\}$ for $n\geq N$. Define 
		\begin{align}
			H_{n,t}=\{z\in H:\sum_{i\in I_n}c_i\chi_{B_{F_{n}}(x_i,\epsilon)}(z)>t\}
		\end{align}
		and
		\begin{align}
			I_n^t=\{i\in I_n:B_{F_{n}}(x_i,\epsilon)\cap H_{n,t}\not=\emptyset\}.
		\end{align}
		It follows that $H_{n,t}\subset \bigcup_{i\in I_n^t} B_{F_{n}}(x_i,\epsilon)$. Let $\mathcal{B}=\{B_{F_{n}}(x_i,\epsilon)\}_{i\in I_n^t}$. By Lemma \ref{cover}, there exists a finite subset $J\subset I^t_n$ such that
		$$\bigcup_{i\in I_n^t} B_{F_{n}}(x_i,\epsilon)\subset \bigcup_{j\in J} B_{F_{n}}(x_j,5\epsilon)$$
		and for any $i,j\in J$ with $i\not= j$, one has $I_n^t(i)\cap I_n^t(j)=\emptyset$, where $$I_n^t(i)=\{j\in I_n^t:B_{F_{n}}(x_j,\epsilon)\cap B_{F_{n}}(x_i,\epsilon)\not=\emptyset\}.$$ For each $j\in J$, choose $y_j\in B_{F_{n}}(x_j,\epsilon)\cap H_{n,t}$. Then $\sum_{i\in I_n^t}c_i\chi_{B_{F_{n}}(x_i,\epsilon)}(y_j)>t$ and hence $\sum_{i\in I_n^t(j)} c_i>t$.

		For each $j\in J$ and $i\in I_n^t(j)$, there exists $z_i\in B_{F_{n}}(x_j,\epsilon)\cap B_{F_{n}}(x_i,\epsilon)$. Thus for each $j\in J$, 
		\begin{align*}
			\sum_{i\in I_n^t(j)} c_i\exp{[-s\Phi_{F_{n}}(z_i)]}&\geq\sum_{i\in I_n^t(j)} c_i\exp{[-s(\Phi_{F_{n}}(x_j)+\frac{|F_{n}|\theta\hat{\Phi}}{4s})]}\\
			&> t\exp{[-s(\Phi_{F_{n}}(x_j)+\frac{|F_{n}|\theta\hat{\Phi}}{4s})]}           \\
			&>t\exp{[-s(\Phi_{F_{n}}(x_j)+\frac{\Phi_{F_{n}}(x_j)}{\hat{\Phi}}\frac{\theta\hat{\Phi}}{4s})]}\\
			&=t\exp{[-s(1+\frac{\theta}{4s})(\Phi_{F_{n}}(x_j))]}.
		\end{align*}
		It follows that 
		\begin{align*}
			\sum_{j\in J}\exp{[-s(\Phi_{F_{n}}(x_j)+\frac{|F_{n}|\theta\hat{\Phi}}{4s})]}&<\frac{1}{t}\sum_{j\in J}\sum_{i\in I_n^t(j)} c_i\exp{[-s\Phi_{F_{n}}(z_i)]}\\
			&\leq \frac{1}{t}\sum_{i\in I_n^t} c_i\exp{[-s(\sup_{y\in B_{F_{n}}(x_i,\epsilon)}\Phi_{F_{n}}(y)-\frac{|F_{n}|\theta\hat{\Phi}}{4s})]}.
		\end{align*}
		This implies that 
		\begin{align*}
			\sum_{j\in J}\exp{[-s(\Phi_{F_{n}}(x_j)+\frac{|F_{n}|\theta\hat{\Phi}}{2s})]}
			\leq \frac{1}{t}\sum_{i\in I_n^t} c_i\exp{[-s(\sup_{y\in B_{F_{n}}(x_i,\epsilon)}\Phi_{F_{n}}(y))]}.
		\end{align*}
		Therefore,
		\begin{align*}
			\widetilde{M}^{s+\theta}_{N,5\epsilon}(H_{n,t},\{F_n\}_{n=1}^{\infty},\Phi)&\leq \sum_{j\in J} \exp{[-(s+\theta)\Phi_{F_{n}}(x_j)]}\\&
			\leq \frac{1}{n^2 t}\sum_{i \in I_n} c_i \exp{[-s\sup_{y\in B_{F_{n}}(x_i,\epsilon)}\Phi_{F_{n}}(y)]}.  
		\end{align*}
		
		By definition, $H=\bigcup_{n\geq N} H_{n,\frac{1}{n^2}t}$. Hence
		\begin{align*}
			\widetilde{M}_{N,5\epsilon}^{s+\theta}(H,\{F_n\}_{n=1}^{\infty},\Phi)&\leq \sum_{n\geq N} \widetilde{M}_{N,5\epsilon}^{s+\theta}(H_{n,\frac{1}{n^2}t},\{F_n\}_{n=1}^{\infty},\Phi)\\
			&\leq \frac{1}{t}\sum_{i\in I}c_i \exp{[-s\sup_{y\in B_{F_{n_i}}(x_i,\epsilon)}\Phi_{F_{n_i}}(y)]}.
		\end{align*}
		Let $t\rightarrow 1$. It follows that $\widetilde{M}^{s+\theta}_{N,\frac{\epsilon}{2}}(H,\Phi)\leq W^{s}_{N,\epsilon}(H,\Phi)$. Letting $N \to \infty$, we have $\widetilde{M}^{s+\theta}_{5\epsilon}(H,\Phi)\leq W^s_{\epsilon}(H,\Phi)\leq M^s_{\epsilon}(H,\Phi).$  
		
		It follows that $\widetilde{M}_{5\epsilon}(H,\{F_n\}_{n=1}^{\infty},\Phi)\leq W_{\epsilon}(H,\{F_n\}_{n=1}^{\infty},\Phi)+\theta\leq M_{\epsilon}(H,\{F_n\}_{n=1}^{\infty},\Phi).$ Taking $\epsilon \to 0$ tend to $0$, by Proposition \ref{inter} and with the  arbitrariness of $\theta>0$, we obtain that  ${dim}_H^{\widetilde{BS}}(\{F_n\}_{n=1}^{\infty},\Phi)={dim}_H^{\widetilde{BS,WB}}(\{F_n\}_{n=1}^{\infty},\Phi)$. 
	\end{proof}

	Now, we are ready to prove the variational principle, here we are inspired by \cite{zhong2023variational,liu2025variational}.
	\begin{theorem}\cite[Theorem 4.1]{liu2025variational}\label{thm3.5}
			Let  $\Phi: X\rightarrow \mathbb{R}$ be a  positive continuous function and  $H$ be a non-empty compact subset of $X$.  Then 
		\begin{align*}
			dim_H^{\widetilde{BS}}(\{F_n\}_{n=1}^{\infty},\Phi)=\sup\{\underline{P}_{\mu}^{\widetilde{BK}}(\{F_n\}_{n=1}^{\infty},\Phi):\mu\in {M}(X),\mu(H)=1\}.
		\end{align*}
	\end{theorem}

	\begin{theorem}\label{main thm1}(=\Cref{m1})
		Let  $\Phi: X\rightarrow \mathbb{R}$ be a  positive continuous function and  $H$ be a non-empty compact subset of $X$.  Then 
		\begin{align*}
			dim_H^{\widetilde{BS}}(\{F_n\}_{n=1}^{\infty},\Phi)&=\sup\{\underline{P}_{\mu}^{\widetilde{BK}}(\{F_n\}_{n=1}^{\infty},\Phi):\mu\in {M}(X),\mu(H)=1\}\\
			&=\sup\{dim_{\mu}^{\widetilde{BS}}(\{F_n\}_{n=1}^{\infty},\Phi):\mu\in {M}(X),\mu(H)=1\}\\
			&=\sup\{dim_{\mu}^{\widetilde{K}}(\{F_n\}_{n=1}^{\infty},\Phi):\mu\in {M}(X),\mu(H)=1\}.
		\end{align*}
	\end{theorem}
	\begin{proof}
		Clearly, we have that 
		$$\sup\{dim_{\mu}^{\widetilde{BS}}(\{F_n\}_{n=1}^{\infty},\Phi):\mu\in {M}(X),\mu(H)=1\}\leq 	dim_H^{\widetilde{BS}}(\{F_n\}_{n=1}^{\infty},\Phi).$$
		Applying Proposition \ref{prop 2.5}, this gives that 
		$$\sup\{\underline{P}_{\mu}^{\widetilde{BK}}(\{F_n\}_{n=1}^{\infty},\Phi):\mu\in {M}(X),\mu(H)=1\}\leq 	\sup\{dim_{\mu}^{\widetilde{K}}(\{F_n\}_{n=1}^{\infty},\Phi):\mu\in {M}(X),\mu(H)=1\}.$$
		Finally, from Theorems \ref{equ} and  \ref{thm3.5}, it follows that 
		 	\begin{align*}
		 	dim_H^{\widetilde{BS}}(\{F_n\}_{n=1}^{\infty},\Phi)&=\sup\{\underline{P}_{\mu}^{\widetilde{BK}}(\{F_n\}_{n=1}^{\infty},\Phi):\mu\in {M}(X),\mu(H)=1\}\\
		 	&=\sup\{dim_{\mu}^{\widetilde{BS}}(\{F_n\}_{n=1}^{\infty},\Phi):\mu\in {M}(X),\mu(H)=1\}\\
		 	&=\sup\{dim_{\mu}^{\widetilde{K}}(\{F_n\}_{n=1}^{\infty},\Phi):\mu\in {M}(X),\mu(H)=1\}.
		 \end{align*}

	\end{proof}
	\begin{remark}

	(1) Theorem \ref{main thm1} yields two new variational principles for BS dimension under group actions, thereby extending and refining previous work under amenable group actions in \cite{liu2025variational}. 
	 
	 	(2) For $\Phi \equiv 1$,  Theorem \ref{main thm1} specializes to recover and strengthen \cite[Theorem 3.1]{zheng2016bowen} on Bowen's topological entropy under amenable group actions.
	\end{remark}

	
	Building on the methodology developed in \cite[Theorem 1.4]{xu2024inverse}, we establish the following inverse variational relationship between BS dimension and BS dimension in the sense of Katok or BS dimension of Borel probability measure.
	
	\begin{theorem}\label{inverse}(=\Cref{m2})
		Let  $\Phi: X\rightarrow \mathbb{R}$ be a  positive continuous function and  $\mu\in {M}(X)$. Then
		$$
		\begin{aligned}
			dim_{\mu}^{\widetilde{BS}}(\{F_n\}_{n=1}^{\infty},\Phi)&=dim_{\mu}^{\widetilde{K}}(\{F_n\}_{n=1}^{\infty},\Phi)\\
			&=  \inf \left\{dim_H^{\widetilde{BS}}(\{F_n\}_{n=1}^{\infty},\Phi):  \mu(H)=1\right\}\\
			&=\lim\limits_{\epsilon\to 0}\inf \left\{{M}_{\epsilon}(H,\{F_n\}_{n=1}^{\infty},\Phi): \mu(H)=1\right\}.
		\end{aligned}
		$$
	\end{theorem}
	
	\begin{proof}
		Let 
		$dim_{\mu}^{\widetilde{BS}}(\{F_n\}_{n=1}^{\infty},\Phi)=\mathfrak{S},$
		then for any  $\lambda>0 $, there exists  $\epsilon_{0}>0$  such that for any  $0<\epsilon<\epsilon_{0}<1 $, it has
		$$
		\lim_{\delta\to 0}\inf\{{M}_{\epsilon}(H,\{F_n\}_{n=1}^{\infty},\Phi):\mu(H) \geq 1-\delta\}<\mathfrak{S}+\frac{\lambda}{3}.$$
		
		Hence, for any $ 0<\epsilon<\epsilon_{0} $, there exists  $\delta(\epsilon)>0 $, such that for any  $0<\delta<\delta(\epsilon) $, one has
		$$
		\inf\{{M}_{\epsilon}(H,\{F_n\}_{n=1}^{\infty},\Phi):\mu(H) \geq 1-\delta\}<\mathfrak{S}+\frac{2\lambda}{3}.
		$$
		
		Therefore, for every  $k$  large enough, there exists $ H_{k} $ with  $\mu\left(H_{k}\right)\geq 1-\frac{1}{k} $, such that
		$$
		{M}_{\epsilon}(H_{k},\{F_n\}_{n=1}^{\infty},\Phi)<\mathfrak{S}+\lambda.
		$$
		
		Take $ H=\cup H_{k} $. Then we have  $\mu(H)=1 $, furthermore, one has
		$$
		{M}_{\epsilon}(H,\{F_n\}_{n=1}^{\infty},\Phi)=\sup_{k} {M}_{\epsilon}(H_{k},\{F_n\}_{n=1}^{\infty},\Phi) \leq \mathfrak{S}+\lambda.
		$$
		
		Thus, we have
		$$
		\inf \left\{{M}_{\epsilon}(H,\{F_n\}_{n=1}^{\infty},\Phi): \mu(H)=1\right\} \leq \mathfrak{S}+\lambda.
		$$
		Hence, 
		\begin{equation*}
			\lim\limits_{\epsilon\to0}\inf \left\{{M}_{\epsilon}(H,\{F_n\}_{n=1}^{\infty},\Phi): \mu(H)=1\right\} \leq \mathfrak{S}=dim_{\mu}^{\widetilde{BS}}(\{F_n\}_{n=1}^{\infty},\Phi).
		\end{equation*}
		In fact,  for any  $\delta>0 $, one has
		$$
		\inf \left\{{M}_{\epsilon}(H,\{F_n\}_{n=1}^{\infty},\Phi): \mu(H)=1\right\} \geq \inf \left\{{M}_{\epsilon}(H,\{F_n\}_{n=1}^{\infty},\Phi): \mu(H)\geq 1-\delta\right\}.
		$$
		
		Therefore,
		$$
		\begin{aligned}
			& \lim\limits_{\epsilon\to0}\inf \left\{{M}_{\epsilon}(H,\{F_n\}_{n=1}^{\infty},\Phi): \mu(H)=1\right\} \\
			\geq &\lim_{\epsilon\to 0}\lim_{\delta\to 0}\inf\{{M}_{\epsilon}(H,\{F_n\}_{n=1}^{\infty},\Phi):\mu(H) \geq 1-\delta\}\\
			= &dim_{\mu}^{\widetilde{BS}}(\{F_n\}_{n=1}^{\infty},\Phi).
		\end{aligned}
		$$
		
		Combining above arguments, we have
		\begin{equation} \label{eq3}
			dim_{\mu}^{\widetilde{BS}}(\{F_n\}_{n=1}^{\infty},\Phi)= \lim\limits_{\epsilon\to0}\inf \left\{{M}_{\epsilon}(H,\{F_n\}_{n=1}^{\infty},\Phi): \mu(H)=1\right\} .
		\end{equation}
		
		Next, using (\ref{eq3}), we can state that
		$$
		dim_{\mu}^{\widetilde{BS}}(\{F_n\}_{n=1}^{\infty},\Phi)=\inf \left\{dim_H^{\widetilde{BS}}(\{F_n\}_{n=1}^{\infty},\Phi):  \mu(H)=1\right\}.
		$$
		
		To see it, for any  $\eta>0 $, there exists  $\epsilon_{0}>0$  such that for any  $0<\epsilon<\epsilon_{0} ,$
		$$
		\inf \left\{{M}_{\epsilon}(H,\{F_n\}_{n=1}^{\infty},\Phi): \mu(H)=1\right\} <\mathfrak{S}+\frac{\eta}{2}.
		$$
		Hence, there exists $ H_{\epsilon} \subset X$  with  $\mu\left(H_{\epsilon}\right)=1 $ such that
		$$
		{M}_{\epsilon}(H_{\epsilon},\{F_n\}_{n=1}^{\infty},\Phi)<\mathfrak{S}+\eta.
		$$
		
		Take  $\epsilon=\frac{1}{j}$. Then for  sufficiently large $j$, there exists  $ H_{j} \subset X $ with  $\mu\left(H_{j}\right)=1$  such that
		$$
		{M}_{\frac{1}{j}}(H_{\frac{1}{j}},\{F_n\}_{n=1}^{\infty},\Phi)<\mathfrak{S}+\eta.$$
		
		Let  $H=\cap_{j} H_{j} $. Then we have  $\mu(H)=1$  and
		$$
		{M}_{\frac{1}{j}}(H,\{F_n\}_{n=1}^{\infty},\Phi)\leq {M}_{\frac{1}{j}}(H_{\frac{1}{j}},\{F_n\}_{n=1}^{\infty},\Phi)<\mathfrak{S}+\eta.
		$$
		It follows that
		$$\inf \left\{dim_H^{\widetilde{BS}}(\{F_n\}_{n=1}^{\infty},\Phi):  \mu(H)=1\right\}\leq {M}_{\frac{1}{j}}(H,\{F_n\}_{n=1}^{\infty},\Phi)\leq \mathfrak{S}+\eta .$$
		Since $\eta$ is arbitrary, then one has
		\begin{equation*}
			\inf \left\{dim_H^{\widetilde{BS}}(\{F_n\}_{n=1}^{\infty},\Phi):  \mu(H)=1\right\}\leq \mathfrak{S}=dim_{\mu}^{\widetilde{BS}}(\{F_n\}_{n=1}^{\infty},\Phi).
		\end{equation*}
		
		Conversely, for each  $H \subset X$  with  $\mu(H)=1 $, by Proposition \ref{equ} one has
		$$
		dim_{\mu}^{\widetilde{BS}}(\{F_n\}_{n=1}^{\infty},\Phi)=dim_{\mu}^{\widetilde{K}}(\{F_n\}_{n=1}^{\infty},\Phi) \leq dim_H^{\widetilde{BS}}(\{F_n\}_{n=1}^{\infty},\Phi),
		$$
		which implies that 
		$$
		dim_{\mu}^{\widetilde{BS}}(\{F_n\}_{n=1}^{\infty},\Phi)=dim_{\mu}^{\widetilde{K}}(\{F_n\}_{n=1}^{\infty},\Phi) \leq \inf \left\{dim_H^{\widetilde{BS}}(\{F_n\}_{n=1}^{\infty},\Phi):  \mu(H)=1\right\}.
		$$
		Hence, we get that 
		$$
		\begin{aligned}
			dim_{\mu}^{\widetilde{BS}}(\{F_n\}_{n=1}^{\infty},\Phi)&=dim_{\mu}^{\widetilde{K}}(\{F_n\}_{n=1}^{\infty},\Phi)\\
			&=  \inf \left\{dim_H^{\widetilde{BS}}(\{F_n\}_{n=1}^{\infty},\Phi):  \mu(H)=1\right\}\\
			&=\lim\limits_{\epsilon\to 0}\inf \left\{{M}_{\epsilon}(H,\{F_n\}_{n=1}^{\infty},\Phi): \mu(H)=1\right\}.
		\end{aligned}
		$$
	\end{proof}

	\subsection{Billingsley’s Theorem for BS dimension under amenable group actions}
Beyond variational principles, another fundamental connection between measure-theoretic entropy and topological entropy is established through Billingsley-type theorem. The dynamical version of this relationship was first developed by Ma and Wen \cite{ma2008billingsley}, who proved that Bowen entropy can be expressed in terms of local entropy, thus providing a perfect analogue of Billingsley's classical result for Hausdorff dimension.
	
	Subsequently, Tang et al. \cite{tang2015variational} generalized the result to Pesin-Pitskel topological pressure.  Huang et al. \cite{huang2019billingsley}  further extended it to amenable group actions.
	The most recent advancement comes from Liu and Peng \cite{liu2025variational}, who established a comprehensive Billingsley-type theorem for BS dimension, more precisely, they extended \cite[Theorem 7.1]{wang2012variational}, \cite[Theorem A]{tang2015variational} and \cite[Theorem 4.8]{huang2019billingsley} to the BS dimension under amenable group actions.
	\begin{theorem}\label{billing-1}\cite[Theorem 3.1]{liu2025variational}
		Let   $\mu \in {M}(X), H \subset X$  and  let $\Phi: X \to \mathbb{R}$  be a positive continuous function. For each  $s \in(0, \infty) $, then
		
		(1) if  $\underline{P}_{\mu}^{\widetilde{BK}}(\{F_n\}_{n=1}^{\infty},\Phi,x) \leq s$  for all  $x \in H $, then  $dim_H^{\widetilde{BS}}(\{F_n\}_{n=1}^{\infty},\Phi) \leq s $;
		
		(2) if  $\underline{P}_{\mu}^{\widetilde{BK}}(\{F_n\}_{n=1}^{\infty},\Phi,x) \geq s$  for all  $x \in H $, and  $\mu(H)>0 $, then  $dim_H^{\widetilde{BS}}(\{F_n\}_{n=1}^{\infty},\Phi) \geq s $.	
		
	\end{theorem}

	\section{BS packing dimension under amenable group actions}\label{packing dimension}
	In this section and next section, we focus on the study of BS packing dimension under amenable group actions, which is a natural generalization of BS packing   dimension in \cite[Definition 2.2]{wang2012variational} and packing topological entropy given in \cite{feng2012variational} and \cite{dou2023packing}.
	
	Let $H\subset X$ be a non-empty subset, $\epsilon>0$, $\Phi \in C(X,\mathbb{R})$ be a positive continuous function, $N\in \mathbb{N}$, $s >0$.
	Put
	\begin{equation}\label{eq5.1}
		L^{s}_{N,\epsilon}(H,\{F_n\}_{n=1}^{\infty},\Phi)=\sup\left\{\sum_i\limits  \exp(-s\sup_{y\in \overline{B}_{F_{n_i}}(x_i,\epsilon)}\Phi_{F_{n_i}}(y))\right\},
	\end{equation}
	where the supremum  is taken over all  finite or countable disjoint $\{\overline B_{F_{n_i}}(x_i,\epsilon)\}_{i\in I}$ with $n_i \geq N,~x_i \in H.$  Since $L^{s}_{N,\epsilon}(H,\{F_n\}_{n=1}^{\infty},\Phi)$ is decreasing when $N$ increases,
	the following limit exists.
	Set
	$$L^{s}_{\epsilon}(H,\{F_n\}_{n=1}^{\infty},\Phi)=\lim_{N\to \infty}L^{s}_{N,\epsilon}(H,\{F_n\}_{n=1}^{\infty},\Phi).$$
	Put 
	\begin{align*}
		\mathscr{L}^{s}_{\epsilon}(H,\{F_n\}_{n=1}^{\infty},\Phi)&=\inf \{{\sum_{i=1}^{\infty}L^{s}_{\epsilon}(H_i,\{F_n\}_{n=1}^{\infty},\Phi):H \subset \cup_{i=1}^{\infty} H_i}\},\\
		\mathscr{L}_{\epsilon}(H,\{F_n\}_{n=1}^{\infty},\Phi)&=\sup{\{s:\mathscr{L}^{s}_{\epsilon}(H,\{F_n\}_{n=1}^{\infty},\Phi)=\infty\}}\\
		&=\inf{\{s:\mathscr{L}^{s}_{\epsilon}(H,\{F_n\}_{n=1}^{\infty},\Phi)=0\}},\\
		dim_H^{\widetilde{BSP}}(\{F_n\}_{n=1}^{\infty},\Phi)&=\lim_{\epsilon \to 0}	\mathscr{L}_{\epsilon}(H,\{F_n\}_{n=1}^{\infty},\Phi).
	\end{align*}
	Since $\mathscr{L}_{\epsilon}(H,\{F_n\}_{n=1}^{\infty},\Phi)$ is increasing when $\epsilon$ decreases, the above limit exists. Then we call $dim_H^{\widetilde{BSP}}(\{F_n\}_{n=1}^{\infty},\Phi)$ {\it BS packing dimension} of the set $H$ along $\{F_n\}_{n=1}^\infty$.
	When $\Phi=1$,  $dim_H^{\widetilde{BSP}}(\{F_n\}_{n=1}^{\infty},1)$ is also the packing topological entropy $h_{top}^{P}(H,\{F_n\}_{n=1}^{\infty})$ given by Dou, Zheng and Zhou \cite{dou2023packing}.
	
	Analogous to Definition \ref{bs},  the notion of  BS packing dimension along $\{F_n\}_{n=1}^{\infty}$ can also be given in an alternative way. Given $B_{F_{n_i}}(x_i,\epsilon)$, we can replace $\sup_{y\in \overline{B}_{F_{n_i}}(x_i,\epsilon)}\Phi_{F_{n_i}}(y)$ by $\Phi_{F_{n_i}}(x_i)$ in equation (\ref{eq5.1}) to give a new definition. We denote by $\widetilde{L}^{s}_{N,\epsilon}(H,\{F_n\}_{n=1}^{\infty},\Phi)$, $\widetilde{\mathscr{L}}_{\epsilon}(H,\{F_n\}_{n=1}^{\infty},\Phi)$ and $\widetilde{dim}_H^{\widetilde{BSP}}(\{F_n\}_{n=1}^{\infty},\Phi)$ the new corresponding quantities of \\ $L^{s}_{N,\epsilon}(H,\{F_n\}_{n=1}^{\infty},\Phi)$, $\mathscr{L}_{\epsilon}(H,\{F_n\}_{n=1}^{\infty},\Phi)$ and $dim_H^{\widetilde{BSP}}(\{F_n\}_{n=1}^{\infty},\Phi)$, respectively.

	\begin{proposition}\label{inter-1-1}
		Let   $H\subset X$ be a non-empty subset and $\Phi: X\rightarrow \mathbb{R}$ be a  positive continuous function. Then $${dim}_H^{\widetilde{BSP}}(\{F_n\}_{n=1}^{\infty},\Phi)=\widetilde{dim}_H^{\widetilde{BSP}}(\{F_n\}_{n=1}^{\infty},\Phi).$$
	\end{proposition}
	\begin{proof}
		It is similar to Proposition \ref{inter}, here we omit it.
		
	\end{proof}

	\begin{proposition}\label{prop 5.2}
		Let $\Phi: X\rightarrow \mathbb{R}$ be a  positive continuous function.  For any $H \subset X,$
		$${dim}_H^{\widetilde{BS}}(\{F_n\}_{n=1}^{\infty},\Phi) \leq {dim}_H^{\widetilde{BSP}}(\{F_n\}_{n=1}^{\infty},\Phi).$$
	\end{proposition}
	
	\begin{proof}
		We follow the idea of \cite[Proposition 2.4(4)]{zhong2023variational} to give the proof.  Suppose that $${dim}_H^{\widetilde{BS}}(\{F_n\}_{n=1}^{\infty},\Phi)>s>0.$$ 
		For any $\epsilon>0$  and  $n\in \mathbb{N}$, let 
		$$\mathcal{F} _{F_n,\epsilon}=\{\mathcal{F}:\mathcal{F}=\{\overline B_{F_n}(x_i,\epsilon)\} \text{ disjoint },x_i \in H \}.$$
		Take $\mathcal{F}({F_n},\epsilon,H) \in \mathcal{F} _{F_n,\epsilon}$ such that $\lvert \mathcal{F}({F_n},\epsilon,H) \rvert=\max_{\mathcal{F} \in \mathcal{F} _{F_n,\epsilon}}{\lvert \mathcal{F} \rvert}.$
		For convenience, we denote $\mathcal{F}({F_n},\epsilon,H)=\{\overline B_{F_n}(x_i,\epsilon):i=1,...,\lvert \mathcal{F}({F_n},\epsilon,H) \rvert\}.$
		It is easy to check that
		$$H \subset \bigcup\limits_ {i=1}^{\lvert \mathcal{F}({F_n},\epsilon,H) \rvert}{B_{F_n}}(x_i,2\epsilon+\delta),\forall \delta>0.$$
		Then for any $s \in \mathbb{R},$
		\begin{align*}
			\widetilde{M}^{s}_{n,2\epsilon+\delta}(H,\{F_n\}_{n=1}^{\infty},\Phi) &\leq \sum_{i=1}^{\lvert \mathcal{F}({F_n},\epsilon,H) \rvert}  e^{-s \Phi_{F_n}(x_i)}\\
			&\leq \widetilde{L}^{s}_{n,\epsilon}(H,\{F_n\}_{n=1}^{\infty},\Phi).
		\end{align*}
		This implies that $$\widetilde{M}^{s}_{2\epsilon+\delta}(H,\{F_n\}_{n=1}^{\infty},\Phi) \leq \widetilde{\mathscr{L}}^{s}_{\epsilon}(H,\{F_n\}_{n=1}^{\infty},\Phi).$$
		Since ${dim}_H^{\widetilde{BS}}(\{F_n\}_{n=1}^{\infty},\Phi)>s>0$, $\widetilde{M}^{s}_{2\epsilon+\delta}(H,\{F_n\}_{n=1}^{\infty},\Phi) \geq 1$ when $\epsilon$  and $\delta$ are small enough, this follows that, $\widetilde{\mathscr{L}}^{s}_{\epsilon}(H,\{F_n\}_{n=1}^{\infty},\Phi) \geq 1.$
		Hence, we get that
		$\widetilde{\mathscr{L}}_{\epsilon}(H,\{F_n\}_{n=1}^{\infty},\Phi) \geq s$ for $\epsilon$ small enough. 
		Therefore, by Propositions \ref{inter} and  \ref{inter-1-1}, one has ${dim}_H^{\widetilde{BSP}}(\{F_n\}_{n=1}^{\infty},\Phi)\geq s$ and  
		${dim}_H^{\widetilde{BSP}}(\{F_n\}_{n=1}^{\infty},\Phi) \geq {dim}_H^{\widetilde{BS}}(\{F_n\}_{n=1}^{\infty},\Phi).$
	\end{proof}
	
		By the theory of Carath{\'e}odory-Pesin structure, the proof of the  following proposition is standard, one can refer to \cite{pesin2008dimension}.
	\begin{proposition} \label{prop 2.211}	
		(1) If $H_1\subset H_2 \subset X$, then $$dim_{H_1}^{\widetilde{BSP}}(\{F_n\}_{n=1}^{\infty},\Phi)\leq dim_{H_2}^{\widetilde{BSP}}(\{F_n\}_{n=1}^{\infty},\Phi).$$
		
		(2) If $H=\cup_{i\geq 1}H_i$ is a union of  sets $H_i\subset X$, then $$dim_{H}^{\widetilde{BSP}}(\{F_n\}_{n=1}^{\infty},\Phi)=\sup_{i\geq 1}dim_{H_i}^{\widetilde{BSP}}(\{F_n\}_{n=1}^{\infty},\Phi).$$
	\end{proposition}

	Analogous to Theorem \ref{bowen} and \cite[Theorem 3.6]{wang2012variational}, we also can prove that the BS packing dimension is the unique root of packing topological pressure function under amenable group actions. Actually, we extend the results of  \cite{wang2012variational}  to packing BS dimension under amenable group actions by below Theorem \ref{packing}.
	
	\begin{theorem}[Bowen’s packing pressure equation] \label{packing}
		For any positive continuous function  $\Phi: X \rightarrow \mathbb{R} $, we have  $dim_H^{\widetilde{BSP}}(\{F_n\}_{n=1}^{\infty},\Phi)=t $, where  $t $ is the unique root of the equation  $P_{H}^{P}(\{F_n\}_{n=1}^{\infty}, -t \Phi)=0 .$
	\end{theorem}
	
	\begin{proof}
		It is similar to Theorem \ref{bowen}, here we leave this proof to readers.	Moreover, one can refer to \cite[Theorem 3.6]{wang2012variational} for similar methods under $\mathbb{Z}_{+}$-actions.
	\end{proof}
	Analogous to subsection \ref{subsec 2.6}, we can also give the notions of BS packing dimension in the sense of Katok.
	\subsection{BS packing dimension in the sense of Katok}
	
	\begin{definition}\label{def 5.5}
		Let 
		$$dim_{\mu}^{\widetilde{BSP}}(\{F_n\}_{n=1}^{\infty},\Phi,\epsilon)=\lim_{\delta\rightarrow 0} \inf\{\widetilde{\mathscr{L}}_{\epsilon}(H,\{F_n\}_{n=1}^{\infty},\Phi):\mu(H)\geq 1-\delta\}.$$
		
		We call the following quantity
		$$dim_{\mu}^{\widetilde{BSP}}(\{F_n\}_{n=1}^{\infty},\Phi):=\lim_{\epsilon \to 0}dim_{\mu}^{\widetilde{BSP}}(\{F_n\}_{n=1}^{\infty},\Phi,\epsilon)$$
		BS packing dimension of $\mu$ along  $\{F_n\}_{n=1}^{\infty}$.
	\end{definition}

	\begin{definition}\label{def 5.6}
		Let $$\Xi^s_{\epsilon}(\mu,\{F_n\}_{n=1}^{\infty},\Phi,\delta)=\inf\{\sum_{i=1}^{\infty}\widetilde{L}^{s}_{\epsilon}(H_i,\{F_n\}_{n=1}^{\infty},\Phi):\mu(\cup_{i=1}^\infty H_i)\geq 1-\delta\}$$ 
		and
		$$\Xi_{\epsilon}(\mu,\{F_n\}_{n=1}^{\infty},\Phi,\delta)=\sup\{s:\Xi^s_{\epsilon}(\mu,\{F_n\}_{n=1}^{\infty},\Phi,\delta)=+\infty\}.$$

		We call the following quantity
		$$dim_{\mu}^{\widetilde{KP}}(\{F_n\}_{n=1}^{\infty},\Phi)=\lim_{\epsilon \to 0}dim_{\mu}^{\widetilde{KP}}(\{F_n\}_{n=1}^{\infty},\Phi,\epsilon)$$ BS packing dimension of $\mu$ in the sense of Katok along $\{F_n\}_{n=1}^{\infty}$, where  $$dim_{\mu}^{\widetilde{KP}}(\{F_n\}_{n=1}^{\infty},\Phi,\epsilon)=\lim_{\delta\rightarrow 0} \Xi_{\epsilon}(\mu,\{F_n\}_{n=1}^{\infty},\Phi,\delta).$$ 
	\end{definition}
	
	\begin{remark}
    Definition \ref{def 5.5} and Definition \ref{def 5.6} are natural generalizations associated with some notions of  packing topological entropy given in \cite{wang2021some}.
	\end{remark}
	
	\begin{proposition}\label{prop 2.10}
		Let $\mu \in \mathcal{M} (X)$,   $\Phi: X \rightarrow \mathbb{R} $ be a  positive continuous function.
		Then, for $\epsilon>0$,
		$$dim_{\mu}^{\widetilde{BSP}}(\{F_n\}_{n=1}^{\infty},\Phi,\epsilon)=dim_{\mu}^{\widetilde{KP}}(\{F_n\}_{n=1}^{\infty},\Phi,\epsilon).$$
		As a direct result, $$dim_{\mu}^{\widetilde{BSP}}(\{F_n\}_{n=1}^{\infty},\Phi)=dim_{\mu}^{\widetilde{KP}}(\{F_n\}_{n=1}^{\infty},\Phi).$$
	\end{proposition}
	
	\begin{proof}
		(1)	We first prove that $dim_{\mu}^{\widetilde{KP}}(\{F_n\}_{n=1}^{\infty},\Phi,\epsilon) \leq dim_{\mu}^{\widetilde{BSP}}(\{F_n\}_{n=1}^{\infty},\Phi,\epsilon).$
		For any $s < dim_{\mu}^{\widetilde{KP}}(\{F_n\}_{n=1}^{\infty},\Phi,\epsilon),$ there exists  $\delta'>0$ such that for any $\delta \in (0,\delta'),$
		$$\Xi_{\epsilon}(\mu,\{F_n\}_{n=1}^{\infty},\Phi,\delta)>s.$$
		Thus
		$$ \Xi_{\epsilon}^{s}(\mu,\{F_n\}_{n=1}^{\infty},\Phi,\delta)=\infty.$$
		If $H\subset \cup_{i=1}^\infty H_i$ with $\mu(H) \geq 1-\delta,$   then $\mu(\cup_{i=1}^\infty H_i) \geq 1-\delta.$
		It follows that
		$$\sum_{i=1}^{\infty}\widetilde{L}^{s}_{\epsilon}(H_i,\{F_n\}_{n=1}^{\infty},\Phi)=\infty,$$
		which implies that $\widetilde{\mathscr{L}}_{\epsilon}^{s}(H,\{F_n\}_{n=1}^{\infty},\Phi)=\infty.$
		Hence, $$\widetilde{\mathscr{L}}_{\epsilon}(H,\{F_n\}_{n=1}^{\infty},\Phi) \geq s$$ and $$dim_{\mu}^{\widetilde{BSP}}(\{F_n\}_{n=1}^{\infty},\Phi,\epsilon) \geq s.$$ Letting $s\to dim_{\mu}^{\widetilde{KP}}(\{F_n\}_{n=1}^{\infty},\Phi,\epsilon), $ this shows that $$dim_{\mu}^{\widetilde{KP}}(\{F_n\}_{n=1}^{\infty},\Phi,\epsilon) \leq dim_{\mu}^{\widetilde{BSP}}(\{F_n\}_{n=1}^{\infty},\Phi,\epsilon).$$
		
		(2)	Next, we shall show the inverse inequality. If $s<dim_{\mu}^{\widetilde{BSP}}(\{F_n\}_{n=1}^{\infty},\Phi,\epsilon),$ then there exists $\delta' >0$ such that for any  $\delta \in (0,\delta'),$
		for any family $\{H_i\}_{i=1}^\infty$ with $\mu(\cup_{i=1}^\infty H_i) \geq 1-\delta,$ we have
		$$\widetilde{\mathscr{L}}_{\epsilon}(\cup_{i=1}^\infty H_i,\{F_n\}_{n=1}^{\infty},\Phi)>s.$$
		This implies that
		$$\widetilde{\mathscr{L}}_{\epsilon}^{s}(\cup_{i=1}^\infty H_i,\{F_n\}_{n=1}^{\infty},\Phi)=\infty.$$
		Thus, one has
		$$\sum_{i=1}^\infty \widetilde{L}_{\epsilon}^{s}( H_i,\{F_n\}_{n=1}^{\infty},\Phi)=\infty,$$
		and
		$$ \Xi^s_{\epsilon}(\mu,\{F_n\}_{n=1}^{\infty},\Phi,\delta)=\infty.$$
		Hence
		$$\Xi_{\epsilon}(\mu,\{F_n\}_{n=1}^{\infty},\Phi,\delta)>s.$$
		Letting $\delta\to0, s\to dim_{\mu}^{\widetilde{BSP}}(\{F_n\}_{n=1}^{\infty},\Phi,\epsilon),$ then we derive that $dim_{\mu}^{\widetilde{BSP}}(\{F_n\}_{n=1}^{\infty},\Phi,\epsilon)\leq dim_{\mu}^{\widetilde{KP}}(\{F_n\}_{n=1}^{\infty},\Phi,\epsilon)$.
		The proof is finished.
	\end{proof}

	Analogous to Theorem \ref{billing-1}, we can establish the Billingsley type theorem of packing BS dimension under amenable group actions. 
	\begin{theorem}\label{billing-2}
		Let  $\{F_n\}_{n=1}^{\infty}$ be a Følner  sequence in $G$ satisfying $\lim\limits_{n \to \infty}\frac{|F_n|}{\log n}=\infty $,  $\mu \in {M}(X), H \subset X$  and  let $\Phi: X \to \mathbb{R}$  be a positive continuous function. For each  $s \in(0, \infty) $, then
		
		(1) if  $\overline{P}_{\mu}^{\widetilde{BK}}(\{F_n\}_{n=1}^{\infty},\Phi,x) \leq s$  for all  $x \in H $, then  $dim_H^{\widetilde{BSP}}(\{F_n\}_{n=1}^{\infty},\Phi) \leq s $;
		
		(2) if  $\overline{P}_{\mu}^{\widetilde{BK}}(\{F_n\}_{n=1}^{\infty},\Phi,x) \geq s$  for all  $x \in H $, and  $\mu(H)>0 $, then  $dim_H^{\widetilde{BSP}}(\{F_n\}_{n=1}^{\infty},\Phi) \geq s $.	
		
	\end{theorem}

	\begin{proof}
		(1)	Fix $\beta>s$, set
		$$
		H_{m}=\left\{x \in H: \limsup _{n \rightarrow \infty} \frac{-\log \mu\left(B_{F_n}(x, \epsilon)\right)}{ \Phi_{F_n}(x)}<\frac{s+\beta}{2}, \text { for any } \epsilon \in\left(0, \frac{1}{m}\right)\right\}.$$
		
		Since  $\overline{P}_{\mu}^{\widetilde{BK}}(\{F_n\}_{n=1}^{\infty},\Phi,x) \leq s$  for all  $x \in H $, then  $H=\bigcup_{m=1}^{\infty} H_{m} .$ Fix  $m \geq 1$  and  $\epsilon \in\left(0, \frac{1}{m}\right) $, for any  $x \in H_{m} $, there exists $N\in \mathbb{N}$ so that for any $n\geq N$, 
		$$
		\mu\left(B_{F_{n}}(x, \epsilon)\right) \geq \exp \left(-\left(\frac{s+\beta}{2}\right) \Phi_{F_n}(x)\right).
		$$
		Let
		$$H_{m,N}=\left\{x \in H_m:\mu(B_{F_n}(x,\epsilon)) \geq  \exp \left(-\left(\frac{s+\beta}{2}\right) \Phi_{F_n}(x)\right),\forall n \geq N,\epsilon \in (0,\frac{1}{m})\right\}.$$
		
		It is clear that $H_m=\bigcup_{N=1}^{\infty}H_{m,N}.$
		Given $\epsilon>0,$ $N\in\mathbb N$ and $J \geq N.$
		Let  $\mathcal{F} =\{B_{F_{n_i}}(x_i,\epsilon)\}_{i \in I},$ where $x_i \in H_{m,N},$ $n_i \geq J $  be a finite or countable disjoint family.
		\begin{align*}
			\sum_{i}e^{-\beta \Phi_{F_{n_i}}(x_i) }&=\sum_{i}e^{-(\frac{\beta+s}{2}+\frac{\beta-s}{2})\Phi_{F_{n_i}}(x_i)}\\
			&\leq e^{-\lvert F_{J} \rvert\hat{\Phi}(\frac{\beta-s}{2})}\sum_{i}e^{-\frac{\beta+s}{2} \Phi_{F_{n_i}}(x_i)}\\
			&\leq e^{-\lvert F_{J} \rvert\hat{\Phi}(\frac{\beta-s}{2})}\sum_{i}\mu(\overline B_{F_{n_i}}(x_i,\epsilon))\\
			&\leq e^{-\lvert F_{J} \rvert\hat{\Phi}(\frac{\beta-s}{2})}.
		\end{align*}
		It follows that 
		$$\widetilde{L}^{\beta}_{J,\epsilon}(H_{m,N},\{F_n\}_{n=1}^{\infty},\Phi) \leq e^{-\lvert F_{J} \rvert\hat{\Phi}(\frac{\beta-s}{2})}.$$
		Letting $J\to\infty,$ we have $\widetilde{L}^{\beta}_{\epsilon}(H_{m,N},\{F_n\}_{n=1}^{\infty},\Phi)=0,$ which implies that $$\widetilde{\mathscr{L}}^{\beta}_{\epsilon}(H_{m},\{F_n\}_{n=1}^{\infty},\Phi)=0.$$
		Hence,  $\widetilde{\mathscr{L}}_{\epsilon}(H_{m},\{F_n\}_{n=1}^{\infty},\Phi) \leq \beta.$
		Letting $\epsilon\to0,$ and by Proposition \ref{inter-1-1}, it follows that $$dim_{H_m}^{\widetilde{BSP}}(\{F_n\}_{n=1}^{\infty},\Phi) \leq \beta.$$
		
		Therefore, $dim_{H}^{\widetilde{BSP}}(\{F_n\}_{n=1}^{\infty},\Phi) \leq \sup_m dim_{H_m}^{\widetilde{BSP}}(\{F_n\}_{n=1}^{\infty},\Phi) \leq\beta,$ 
		Letting $\beta\to s,$ then we get that  $dim_H^{\widetilde{BSP}}(\{F_n\}_{n=1}^{\infty},\Phi) \leq s $.
		
		(2) 
		Fix $\beta <s.$ Let $\delta=\frac{s-\beta}{2}$
		and
		$$H_m=\{x \in H:\limsup_{n \to \infty}\frac{-\log\mu(B_{F_n}(x,\frac{1}{m}))}{\Phi_{F_n}(x)} > \beta + \delta\}.$$
		Then $H=\bigcup_{m=1}^{\infty}H_m.$
		Since $\mu(H)>0$ and $H_n\subset H_{n+1}, n\in\mathbb N,$ there exists $m\in\mathbb N$ such that $\mu(H_m)>0.$ For any  $\epsilon \in (0,\frac{1}{m})$ and $ x \in H_m$, we have
		$$ \limsup_{n \to \infty}\frac{-\log\mu(B_{F_n}(x,\epsilon))}{\Phi_{F_n}(x)}>\beta +\delta.$$
		Next we claim that $\widetilde{\mathscr{L}}^{s}_{\frac{\epsilon}{10}}(H_{m},\{F_n\}_{n=1}^{\infty},\Phi)=\infty,$ also from Proposition \ref{inter-1-1}, this implies that 
		\begin{align*}
			dim_H^{\widetilde{BSP}}(\{F_n\}_{n=1}^{\infty},\Phi) &\geq dim_{H_m}^{\widetilde{BSP}}(\{F_n\}_{n=1}^{\infty},\Phi)\\ &\geq \widetilde{\mathscr{L}}_{\frac{\epsilon}{10}}(H_{m},\{F_n\}_{n=1}^{\infty},\Phi)\\ &\geq s.
		\end{align*}
		To this end, it suffices to show that $\widetilde{L}^{s}_{\frac{\epsilon}{10}}(E,\{F_n\}_{n=1}^{\infty},\Phi)=\infty$ for any Borel subset $E \subset H_m$ with $\mu (E)>0.$
		In fact, for  $E\subset H_m$ with $\mu (E)>0,$
		let 
		$$E_n=\{x\in E:\mu(B_{F_n}(x,\epsilon))<e^{-(\beta+\delta)\Phi_{F_{n}}(x)}\},~n \in \mathbb{N}.$$
		It is clear that $E=\bigcup_{n=N}^{\infty}E_n$ for each $N \in \mathbb{N}. $ Then $\mu(\bigcup_{n=N}^{\infty}E_n)=\mu(E).$
		Hence there exists $n \geq N$ such that
		$$\mu(E_n) \geq \frac{1}{n(n+1)}\mu(E).$$
		Fix such $n$ and let $\mathcal{B} =\{\overline B_{F_n}(x,\frac{\epsilon}{10}):x\in E_n\}.$
		By Lemma \ref{cover}, there exists a finite or countable pairwise disjoint family $\{\overline B_{F_n}(x_i,\frac{\epsilon}{10})\}_{i\in I}$ such that
		$$E_n \subset \bigcup_{x \in E_n}\overline B_{F_n}(x,\frac{\epsilon}{10}) \subset \bigcup_{i\in I} \overline B_{F_n}(x_i,\frac{\epsilon}{2}) \subset \bigcup_{i\in I} \overline B_{F_n}(x_i,\epsilon).$$
		Hence,
		\begin{align*}
			\widetilde{L}^{\beta}_{N,\frac{\epsilon}{10}}(E,\{F_n\}_{n=1}^{\infty},\Phi) &\geq \widetilde{L}^{\beta}_{N,\frac{\epsilon}{10}}(E_n,\{F_n\}_{n=1}^{\infty},\Phi)\\ &\geq \sum_{{i\in I}}e^{-\beta\Phi_{F_n}(x_i)}\\
			&\geq  e^{\lvert F_{n} \rvert\hat{\Phi} \delta}\sum_{{i\in I}}e^{-(\beta+\delta)\Phi_{F_n}(x_i)}\\ &=  e^{\lvert F_{n} \rvert\hat{\Phi} \delta}\sum_{{i\in I}}e^{-(\beta+\delta)\Phi_{F_n}(x_i)}\\&\geq e^{\lvert F_{n} \rvert\hat{\Phi} \delta}\sum_{{i\in I}}\mu(B_{F_n}(x_i,\epsilon))\\
			&\geq e^{\lvert F_{n} \rvert \hat{\Phi}\delta}\mu(E_n)\\ 
			&\geq \frac{e^{\lvert F_{n} \rvert \hat{\Phi}\delta}}{n(n+1)}\mu(E).
		\end{align*}	
		Since $\lim\limits_{n \to \infty}\frac{|F_n|}{\log n}=\infty $ and $\mu(E)>0,$ we have
		$$\widetilde{L}^{\beta}_{\frac{\epsilon}{10}}(E,\{F_n\}_{n=1}^{\infty},\Phi)=\infty.$$ This finishes the proof.
		
	\end{proof}

	\begin{proposition}\label{prop 2.13}
		Let $\mu \in \mathcal{M} (X)$, $H \subset X,$   $\Phi: X \rightarrow \mathbb{R} $ be a  positive continuous function.
		We have $$\overline{P}_{\mu}^{\widetilde{BK}}(\{F_n\}_{n=1}^{\infty},\Phi) \leq dim_{\mu}^{\widetilde{KP}}(\{F_n\}_{n=1}^{\infty},\Phi).$$
	\end{proposition}
	\begin{proof}
		For any $s < \overline{P}_{\mu}^{\widetilde{BK}}(\{F_n\}_{n=1}^{\infty},\Phi),$ we can find a Borel set $A \subset X$ with $\mu(A)>0$ such that, for any $x \in A,$
		$$\lim_{\epsilon \to 0}\limsup_{n \to \infty}\frac{-\log\mu(B_{F_n}(x,\epsilon))}{\Phi_{F_n}(x)} >s.$$
		Given $\delta \in (0,\mu(A))$ and $\epsilon>0.$ We shall show that $dim_{\mu}^{\widetilde{KP}}(\{F_n\}_{n=1}^{\infty},\Phi,\frac{\epsilon}{10})> s,$ which implies that $dim_{\mu}^{\widetilde{KP}}(\{F_n\}_{n=1}^{\infty},\Phi) \geq s.$
		It suffices to show that
		$$\Xi^s_{\frac{\epsilon}{10}}(\mu,\{F_n\}_{n=1}^{\infty},\Phi,\delta)=\infty.$$
		Let $\{H_i\}_{{i \in I}}$ be a finite or countable family with $\mu(\cup_{i \in I} H_i)>1-\delta.$
		Since
		$$A=(A\cap (\cup_{i \in I}H_i))\cup(A\verb|\|\cup_{i \in I}H_i),$$
		it follows that $\mu(A\cap (\cup_{i \in I}H_i)) \geq \mu(A)-\delta>0.$ Thus there exists $i$ such that $\mu(A\cap H_i)>0.$
		Due to Proposition \ref{billing-2}, we have
		$$\widetilde{L}^{s}_{\frac{\epsilon}{10}}(H_i,\{F_n\}_{n=1}^{\infty},\Phi)  \geq \widetilde{L}^{s}_{\frac{\epsilon}{10}}(A\cap H_i,\{F_n\}_{n=1}^{\infty},\Phi)=\infty.$$
		Thus 
		$$\Xi^s_{\frac{\epsilon}{10}}(\mu,\{F_n\}_{n=1}^{\infty},\Phi,\delta)=\infty.$$

		\section{Variational principle for BS packing dimension under amenable group actions}
		In this section, we shall prove the variational principle for BS packing dimension under amenable group actions, i.e., Theorem \ref{thm 1.1}, and the proof  is divided into two parts: upper bound and lower bound.
		
		\begin{theorem}\label{thm 1.1}(=\Cref{m3})
			Let $\{F_n\}_{n=1}^{\infty}$ be a Følner  sequence in G satisfying $\lim_{n \to \infty}\frac{|F_n|}{\log n}=\infty.$ Then for any non-empty analytic subset $H$ of $X$ and a positive continuous function  $\Phi: X \rightarrow \mathbb{R} $.
			Then 
			\begin{align*}
				dim_H^{\widetilde{BSP}}(\{F_n\}_{n=1}^{\infty},\Phi)&=\sup\{\overline{P}_{\mu}^{\widetilde{BK}}(\{F_n\}_{n=1}^{\infty},\Phi):\mu \in \mathcal{M} (X),\mu(H)=1\}\\
				&=\sup\{dim_{\mu}^{\widetilde{BSP}}(\{F_n\}_{n=1}^{\infty},\Phi,\epsilon):\mu \in \mathcal{M} (X),\mu(H)=1\}\\
				&=\sup\{dim_{\mu}^{\widetilde{KP}}(\{F_n\}_{n=1}^{\infty},\Phi):\mu \in \mathcal{M} (X),\mu(H)=1\}.
			\end{align*}
		\end{theorem}
		
	\end{proof}
	\subsection{Lower bound}
	Using Proposition \ref{prop 2.10} and Proposition \ref{prop 2.13}, this shows that
	\begin{align*}
		&\sup\{\overline{P}_{\mu}^{\widetilde{BK}}(\{F_n\}_{n=1}^{\infty},\Phi):\mu \in \mathcal{M} (X),\mu(H)=1\}\\
		\leq& \sup\{dim_{\mu}^{\widetilde{KP}}(\{F_n\}_{n=1}^{\infty},\Phi):\mu \in \mathcal{M} (X),\mu(H)=1\}\\
		=& \sup\{dim_{\mu}^{\widetilde{BSP}}(\{F_n\}_{n=1}^{\infty},\Phi):\mu \in \mathcal{M} (X),\mu(H)=1\}\\
		\leq &dim_H^{\widetilde{BSP}}(\{F_n\}_{n=1}^{\infty},\Phi).
	\end{align*}
	\subsection{Upper bound}

	\begin{lemma}\label{lem 3.1}
		Let $H \subset X,$ $\epsilon,s >0$. If $L^{s}_{\epsilon}(H,\{F_n\}_{n=1}^{\infty},\Phi)=\infty,$ then for any given finite interval
		$(a,b) \subset [0,+\infty)$ and $N \in \mathbb{N}, $ there exists a finite disjoint collection $\{\overline B_{F_{n_i}}(x_i,\epsilon)\}$ such that $x_i \in H,$ $n_i \geq N,$
		and $\sum_{i}e^{-s\Phi_{F_{n_i}}(x_i)} \in (a,b).$
	\end{lemma}
	\begin{proof}
		Take $N_1>N$ large enough such that $e^{|F_{N_1}|\hat{\Phi}s}<b-a.$ Since $L^{s}_{\epsilon}(H,\{F_n\}_{n=1}^{\infty},\Phi)=\infty,$ it follows that $L^{s}_{N_1,\epsilon}(H,\{F_n\}_{n=1}^{\infty},\Phi)=\infty.$
		There hence exists a finite disjoint collection $\{\overline B_{F_{n_i}}(x_i,\epsilon)\}$ such that $x_i \in H,$ $n_i \geq N_1$ and $\sum_{i}e^{-s\Phi_{F_{n_i}}(x_i)}>b.$
		Since $e^{-s\Phi_{F_{n_i}}(x_i)} \leq e^{|F_{N_1}|\hat{\Phi}s} \leq b-a,$ we can discard elements in this collection one by one until we have $\sum_{i}e^{-s\Phi_{F_{n_i}}(x_i)}  \in (a,b).$
	\end{proof}
	We now turn to show the Upper bound. We employ the approach used by Joyce and Preiss \cite{joyce1995existence} and Feng and Huang in \cite{feng2012variational}. Let $H \subset X$ be analytic with $dim_H^{\widetilde{BSP}}(\{F_n\}_{n=1}^{\infty},\Phi)>0$. For any $dim_H^{\widetilde{BSP}}(\{F_n\}_{n=1}^{\infty},\Phi)>s>0$, there exists a compact set $K \subset H$ and $\mu\in \mathcal{M}(K)$ such that $\overline{P}_{\mu}^{\widetilde{BK}}(\{F_n\}_{n=1}^{\infty},\Phi)\geq s$.

	Since $H$ is analytic, there exists a continuous surjective map $\phi :\mathcal{N} \to H.$ Let  $\Gamma _{n_1,n_2,...,n_p}=\{(m_1,m_2,...)\in\mathcal{N}:m_1 \leq n_1,m_2 \leq n_2,...,m_p \leq n_p \} $
	and let $H_{n_1,...,n_p}=\phi(\Gamma _{n_1,n_2,...,n_p}) .$ 
	Take $\epsilon>0$ small enough so that $\mathscr{L}_{\epsilon}(H,\{F_n\}_{n=1}^{\infty},\Phi)>s>0$. Take $t$ with $s<t<\mathscr{L}_{\epsilon}(H,\{F_n\}_{n=1}^{\infty},\Phi)$.
	
	The construction is divided into the following  several steps:
	
	{\bf Step 1.} Construct $K_1,$ $\mu_1,$ $n_1,$ $\gamma_1,$ and  $m_1(·) .$\\
	Note that $	\mathscr{L}^{t}_{\epsilon}(H,\{F_n\}_{n=1}^{\infty},\Phi)=\infty.$ Let
	$$Q=\bigcup\{U \subset X:U~is~open,\mathscr{L}^{t}_{\epsilon}(H\cap U,\{F_n\}_{n=1}^{\infty},\Phi)=0\}.$$
	Then $\mathscr{L}^{t}_{\epsilon}(H\cap Q,\{F_n\}_{n=1}^{\infty},\Phi)= 0$ by the separability of $X$. Let 
	$$H'=H\verb|\|Q=H\cap(X\verb|\|Q).$$
	For any open set $U \subset X,$ either $H'\cap U=\emptyset$ or $\mathscr{L}^{t}_{\epsilon}(H'\cap U,\{F_n\}_{n=1}^{\infty},\Phi)>0.$
	Indeed, suppose that $\mathscr{L}^{t}_{\epsilon}(H'\cap U,\{F_n\}_{n=1}^{\infty},\Phi)=0.$
	Since $H=H'\cup (H\cap Q),$
	\begin{align*}
		\mathscr{L}^{t}_{\epsilon}(H\cap U,\{F_n\}_{n=1}^{\infty},\Phi) &\leq \mathscr{L}^{t}_{\epsilon}(H'\cap U,\{F_n\}_{n=1}^{\infty},\Phi)\\ &+\mathscr{L}^{t}_{\epsilon}(H\cap Q,\{F_n\}_{n=1}^{\infty},\Phi)\\&=0.
	\end{align*}
	Thus $U \subset Q,$ which implies that $H'\cap U=\emptyset.$
	Since
	$$\mathscr{L}^{t}_{\epsilon}(H,\{F_n\}_{n=1}^{\infty},\Phi) \leq \mathscr{L}^{t}_{\epsilon}(H\cap H,\{F_n\}_{n=1}^{\infty},\Phi)+\mathscr{L}^{t}_{\epsilon}(H',\{F_n\}_{n=1}^{\infty},\Phi) $$
	and $\mathscr{L}^{t}_{\epsilon}(H\cap Q,\{F_n\}_{n=1}^{\infty},\Phi)=0,$ we have
	$$\mathscr{L}^{t}_{\epsilon}(H',\{F_n\}_{n=1}^{\infty},\Phi)=\mathscr{L}^{t}_{\epsilon}(H,\{F_n\}_{n=1}^{\infty},\Phi)=\infty.$$
	Using Lemma \ref{lem 3.1}, we can find a finite set $K_1 \subset H',$ an integer-valued function $m_1(x)$ on $K_1$ such that the collection $ \{\overline B_{F_{m_1(x)}}(x,\epsilon)\}_{x \in K_1}$ is disjoint and
	$$\sum_{x \in K_1}  e^{-s\Phi_{F_{m_1(x)}}(x)} \in (1,2).$$
	Define
	$$\mu_1=\sum_{x \in K_1} e^{-s\Phi_{F_{m_1(x)}}(x)}\delta_x,$$ where $\delta_x$ denotes the Dirac measure at $x.$

	Take a small $\gamma_1$ such that for any function $z:K_1 \to X$ with $d(x,z(x))\leq  \gamma_1,$
	we have for each $x \in K_1,$
	$$(\overline B(z(x),\gamma_1)\cup\overline B_{F_{m_1(x)}}(z(x),\epsilon))\bigcap (\bigcup_{y \in K_1\verb|\|\{x\}}\overline B(z(y),\gamma_1)\cup \overline B_{F_{m_1(y)}}(z(y),\epsilon))=\emptyset.$$
	It follows from $K_1 \subset H'$ that for any $x \in K_1,$
	\begin{align*}
		\mathscr{L}^{t}_{\epsilon}(H \cap B(x,\frac{\gamma_1}{4}),\{F_n\}_{n=1}^{\infty},\Phi) &\geq \mathscr{L}^{t}_{\epsilon}(H' \cap B(x,\frac{\gamma_1}{4}),\{F_n\}_{n=1}^{\infty},\Phi)>0.
	\end{align*}
	Therefore, we can pick a sufficiently large $n_1 \in \mathbb{N}$ so that $K_1 \subset H_{n_1}$ and $\mathscr{L}^{t}_{\epsilon}(H_{n_1} \cap B(x,\frac{\gamma_1}{4}),\{F_n\}_{n=1}^{\infty},\Phi)>0$ for each $x \in K_1.$
	
	{\bf Step 2.} Construct $K_2,$ $\mu_2,$ $n_2,$ $\gamma_2,$ and  $m_2(·) .$\\
	The family of balls $\{\overline B(x,\gamma_1)\}_{x \in K_1}$ are pairwise disjoint. For each $x \in K_1,$ since $\mathscr{L}^{t}_{\epsilon}(H_{n_1} \cap B(x,\frac{\gamma_1}{4}),\{F_n\}_{n=1}^{\infty},\Phi)>0,$
	we can construct a finite set  as in  Step 1  
	$$E_2(x) \subset H_{n_1} \cap B(x,\frac{\gamma_1}{4})$$
	and an integer-valued function
	$$m_2:E_2(x) \to \mathbb{N}\cap[\max\{m_1(y):y \in K_1\},\infty)$$
	such that
	\begin{flushleft}
		(2-a) $\mathscr{L}^{t}_{\epsilon}(H_{n_1} \cap U,\{F_n\}_{n=1}^{\infty},\Phi)>0 ,$ for any open set $U$ with $U\cap E_2(x) \neq \emptyset;$\\
		(2-b) the elements in $\{\overline B_{F_{m_2(y)}}(y,\epsilon)\}_{y \in E_2(x)}$ are disjoint, and
	\end{flushleft}
	$$\mu_1(\{x\})<\sum_{y \in E_2(x)}e^{-s\Phi_{F_{m_2(y)}}(y)}<(1+2^{-2})\mu_1(\{x\}).$$
	Actually, we fix $x \in K_1.$
	Denote $F=H_{n_1} \cap B(x,\frac{\gamma_1}{4}).$ 
	Let
	$$Q_x=\bigcup\{U \subset X:U~is~open,\mathscr{L}^{t}_{\epsilon}(F\cap U,\{F_n\}_{n=1}^{\infty},\Phi)=0\}.$$
	Set 
	$$F'=F\verb|\|Q_x.$$
	Then as in {Step 1,} we can show that
	$$\mathscr{L}^{t}_{\epsilon}(F',\{F_n\}_{n=1}^{\infty},\Phi)=\mathscr{L}^{t}_{\epsilon}(F,\{F_n\}_{n=1}^{\infty},\Phi)>0$$
	and
	$$\mathscr{L}^{t}_{\epsilon}(F'\cap U,\{F_n\}_{n=1}^{\infty},\Phi)>0,$$
	for any open set $U$ with $U \cap F'\neq \emptyset.$ Since $s < t,$
	$$\mathscr{L}^{s}_{\epsilon}(F',\{F_n\}_{n=1}^{\infty},\Phi)=\infty.$$
	Using Lemma \ref{lem 3.1} again, we can find a finite set $ E_2(x) \subset F',$ an integer-valued function $m_2(x)$ on $E_2(x)$ so that (2-b) holds.
	Observe that if $U \cap E_2(x)\neq \emptyset$ and $U$ is open, then $U \cap F'\neq \emptyset.$ Hence
	$$\mathscr{L}^{t}_{\epsilon}(H_{n_1}\cap U,\{F_n\}_{n=1}^{\infty},\Phi) \geq \mathscr{L}^{t}_{\epsilon}(F'\cap U,\{F_n\}_{n=1}^{\infty},\Phi)>0.$$
	Then (2-a) holds.
	Since the family $\{\overline{B}(x,\gamma_1)\}_{x \in K_1}$ is disjoint, $E_2(x) \cap E_2(x')=\emptyset$ for different $x,x' \in K_1.$
	Define $K_2=\bigcup_{x \in K_1}E_2(x)$
	and
	$$\mu_2=\sum_{y \in K_2}e^{-s\Phi_{F_{m_2(y)}}(y)}\delta_y.$$
	The elements in  $\{\overline B_{F_{m_2(y)}}(y,\epsilon)\}_{y \in K_2} $ are disjoint. Hence we can take $\gamma_2 \in (0,\frac{\gamma_1}{4})$ such that
	for any function $z:K_2 \to X$ satisfying $d(x,z(x))<\gamma_2,$ for  $x \in K_2.$
	we have
	$$(\overline B(z(x),\gamma_2)\cup \overline B_{F_{m_2(x)}}(z(x),\epsilon))\bigcap (\bigcup_{y \in K_2\verb|\|\{x\}}\overline B(z(y),\gamma_2)\cup \overline B_{F_{m_2(y)}}(z(y),\epsilon))=\emptyset$$
	for each $x \in K_2.$
	Choose a sufficiently large $n_2 \in \mathbb{N}$ so that $K_2 \subset H_{n_1,n_2}$ and 
	$$\mathscr{L}^{t}_{\epsilon}(H_{n_1,n_2} \cap B(x,\frac{\gamma_2}{4}),\{F_n\}_{n=1}^{\infty},\Phi)>0,$$
	for each $x \in K_2.$
	
	{\bf Step 3.} Assume that $K_i,$ $\mu_i,$ $n_i,$ $\gamma_i,$ and  $m_i(·)$ have been constructed for $i=1,...,p.$
	In particular, suppose that for any function $z:K_p \to X$ with $d(x,z(x))<\gamma_p$ for $x\in K_p$,
	we have
	$(\overline B(z(x),\gamma_p)\cup \overline B_{F_{m_p(x)}}(z(x),\epsilon))\cap (\bigcup_{y \in K_p\verb|\|\{x\}}\overline B(z(y),\gamma_p)\cup \overline B_{F_{m_p(y)}}(z(y),\epsilon))=\emptyset,$
	for each $x \in K_p.$
	Then $K_p \subset H_{n_1...n_p}$ and 
	$$\mathscr{L}^{t}_{\epsilon}(H_{n_1...n_p} \cap B(x,\frac{\gamma_p}{4}),\{F_n\}_{n=1}^{\infty},\Phi)>0,$$
	for each $x \in K_p.$
	The family of balls $\{\overline B(x,\gamma_p)\}_{x \in K_p} $ are pairwise disjoint. For each $x \in K_p,$ since $\mathscr{L}^{t}_{\epsilon}(H_{n_1...n_p} \cap B(x,\frac{\gamma_p}{4}),\{F_n\}_{n=1}^{\infty},\Phi)>0,$
	we can construct as in Step 2 a finite set 
	$$E_{p+1}(x) \subset H_{n_1...n_p} \cap B(x,\frac{\gamma_p}{4})$$
	and an integer-valued function
	$$m_{p+1}:E_{p+1}(x) \to \mathbb{N}\cap[\max\{m_p(y):y \in K_p\},\infty)$$
	such that
	\begin{flushleft}
		(3-a) $\mathscr{L}^{t}_{\epsilon}(H_{n_1...n_p} \cap U,\{F_n\}_{n=1}^{\infty},\Phi)>0,$ for any open set $U$ with $U\cap E_{p+1}(x) \neq \emptyset;$\\
		(3-b) the elements in $\{\overline B_{F_{m_{p+1}(y)}}(y,\epsilon)\}_{y \in E_{p+1}(x)}$ are disjoint, and
	\end{flushleft}
	$$\mu_p(\{x\})<\sum_{y \in E_{p+1}(x)}e^{-s \Phi_{F_{m_{p+1}(y)}}(y)}<(1+2^{-p-1})\mu_p(\{x\}).$$
	Clearly  $E_{p+1}(x) \cap E_{p+1}(y) = \emptyset$
	for different $x,y \in K_p.$
	Define $K_{p+1}=\bigcup_{x \in K_p}E_{p+1}(x)$
	and
	$$\mu_{p+1}=\sum_{y \in K_{p+1}}e^{-s\Phi_{F_{m_{p+1}(y)}}(y)}\delta_y.$$
	The elements in  $\{\overline B_{F_{m_{p+1}(y)}}(y,\epsilon)\}_{y \in K_{p+1}} $ are disjoint. Hence we can take $\gamma_{p+1} \in (0,\frac{\gamma_p}{4})$ such that
	for any function $z:K_{p+1} \to X$ satisfying $d(x,z(x))<\gamma_{p+1},$
	we have for each $x \in K_{p+1}$,
	$$(\overline B(z(x),\gamma_{p+1})\cup \overline B_{F_{m_{p+1}(x)}}(z(x),\epsilon)) \\ \bigcap (\bigcup_{y \in K_{p+1}\verb|\|\{x\}}\overline B(z(y),\gamma_{p+1})\cup \overline B_{F_{m_{p+1}(y)}}(z(y),\epsilon))=\emptyset$$
	
	Choose a sufficiently large $n_{p+1} \in \mathbb{N}$ so that $K_{p+1} \subset H_{n_1...n_{p+1}}$ and 
	$$\mathscr{L}^{t}_{\epsilon}(H_{n_1...n_{p+1}} \cap B(x,\frac{\gamma_{p+1}}{4}),\{F_n\}_{n=1}^{\infty},\Phi)>0,$$
	for each $x \in K_{p+1}.$
	As in above steps, we can construct by introduction $\{K_i\},\{\mu_i\},n_i,\gamma_i,and~m_i(·) .$
	We summarize some of their basic properties as follows:
	
	(a)For each $i$, the family $\mathcal{F}_i:=\{\overline{B}(x,\gamma_i):x \in K_i\}$ is disjoint. For every $B \in \mathcal{F}_{i+1},$ there exists $x \in K_i$ such that $B \subset \overline{B}(x,\frac{\gamma_i}{2}).$ 
	
	(b)For each $x \in K_i$ and $z\in \overline{B}(x,\gamma_i),$ we have
	$$\overline B_{F_{m_{i}(x)}}(z,\epsilon) \cap \bigcup_{y \in K_i\verb|\|\{x\}}  \overline B(y,\gamma_i)=\emptyset$$
	and
	\begin{align*}
		\mu_i( \overline{B}(x,\gamma_i))&=e^{-s\Phi_{F_{m_i(x)}(x)}}<\sum_{y \in E_{i+1}(x)}e^{-s\Phi_{F_{m_{i+1}(y)}}(y)}\\
		&<(1+2^{-i-1})\mu_i( \overline{B}(x,\gamma_i)),
	\end{align*}
	where $E_{i+1}(x)=B(x,\gamma_i) \cap K_{i+1}.$
	Furthermore, for $F_i\in \mathcal{F}_{i}$, one has
	\begin{align*}
		\mu_i(F_i) \leq \mu_{i+1}(F_i)&=\sum_{F\in \mathcal{F}_{i+1}:F \subset F_i }\mu_{i+1}(F)\leq \sum_{F\in \mathcal{F}_{i+1}:F \subset F_i }(1+2^{-i-1})\mu_i(F)\\
		&=(1+2^{-i-1})\sum_{F\in \mathcal{F}_{i+1}:F \subset F_i }\mu_i(F)\\
		&\leq (1+2^{-i-1})\mu_i(F_i).
	\end{align*}
	
	Using the above inequalities repeatedly, we have for any $j>i,$ $F_i \in \mathcal{F}_i,$ 
	\begin{align}\label{3.1}
		\mu_i(F_i) \leq \mu_j(F_i) \leq \prod _{n=i+1}^j (1+2^{-n})\mu_i(F_i) \leq C\mu_i(F_i),
	\end{align}
	where $C:=\prod _{n=1}^\infty(1+2^{-n})<\infty.$
	Let $\widetilde\mu$ be a limit point of $\{\mu_i\}$ in the weak-$*$ topology, let
	$$K=\bigcap_{n=1}^{\infty}\overline{\bigcup_{i \geq n}K_i}.$$
	
	Then, $\widetilde\mu$ is supported on $K$ and $K \subset \bigcap_{p=1}^{+\infty}\overline{H_{n_1,...,n_p}}.$
	By the continuity of $\phi,$ applying the Cantor's diagonal argument, we can show that $$\bigcap_{p=1}^{+\infty}\overline{H_{n_1,...,n_p}} =\bigcap_{p=1}^{+\infty}{H_{n_1,...,n_p}}.$$
	Hence, $K$ is a compact subset of $H$. For any $x \in K_i,$ by (\ref{3.1}), 
	\begin{align*}
		e^{-s\Phi_{F_{m_i(x)}(x)}}&=\mu_i(\overline{B}(x,\gamma_i))\leq \widetilde\mu( \overline{B}(x,\gamma_i))\\
		&\leq C\mu_i(\overline{B}(x,\gamma_i))=Ce^{-s\Phi_{F_{m_i(x)}(x)}}.
	\end{align*}
	In particular,
	\begin{align*}
		1\leq \sum_{x \in K_1}\mu_1({B}(x,\gamma_1)) 
		&\leq\widetilde\mu(K)\leq \sum_{x \in K_1}C\mu_1(B(x,\gamma_1))\leq 2C.
	\end{align*}
	For every $x \in K_i$ and $z \in \overline{B}(x,\gamma_i),$
	$$\widetilde\mu(\overline B_{F_{m_i(x)}}(z,\epsilon))\leq \widetilde \mu(\overline B(x,\frac{\gamma_i}{2})) \leq Ce^{-s\Phi_{F_{m_i(x)}(x)}}.$$
	For each $z \in K$ and $i \in \mathbb{N},$ $z \in \overline B(x,\frac{\gamma_i}{2})$ for some $x \in K_i$. Thus
	$$\widetilde\mu(\overline B_{F_{m_i(x)}}(z,\epsilon)) \leq Ce^{-s\Phi_{F_{m_i(x)}(x)}}.$$
	Let $\mu=\widetilde\mu/\widetilde\mu(K).$ Then $\mu \in \mathcal{M} (K),$ $\mu(K)=1,$ and for every $z \in K,$ 
	there exists a sequence $\{k_i\}$ with $k_i \to \infty$ such that
	$$\mu(B_{F_{k_i}}(z,\epsilon)) \leq \frac{Ce^{-s\Phi_{F_{k_i}(z)}}}{\widetilde\mu(K)}.$$
	This implies that $\overline P_\mu(\{F_n\}_{n=1}^{\infty},f) \geq s.$

	\begin{remark}
	(1)	If $\Phi=1$,  Theorem \ref{thm 1.1} coincides with \cite[Theorem 1.3]{dou2023packing}, so we  enhance some  results about  packing topological entropy under amenable group actions.
	
	(2) Naturally, Theorem \ref{thm 1.1} covers the results established by Wang and Chen for the BS packing dimension under $\mathbb{Z}_{+}$-actions in \cite[Theorem 3.12:(1)]{wang2012variational}.
	\end{remark}
	
	Analogous to inverse variational principle for BS dimension under amenable group actions, i.e. Theorem \ref{inverse}, we can also prove the inverse variational principle for BS packing dimension under amenable group actions by a similar method in Theorem \ref{inverse}, here we leave it to readers and give the result directly. 
	
	\begin{theorem}\label{inverse-packing}(=\Cref{m4})
		Let  $\Phi: X\rightarrow \mathbb{R}$ be a  positive continuous function and  $\mu\in {M}(X)$. Then
		$$
		\begin{aligned}
			dim_{\mu}^{\widetilde{BSP}}(\{F_n\}_{n=1}^{\infty},\Phi)&=dim_{\mu}^{\widetilde{KP}}(\{F_n\}_{n=1}^{\infty},\Phi)\\
			&=  \inf \left\{dim_H^{\widetilde{BSP}}(\{F_n\}_{n=1}^{\infty},\Phi):  \mu(H)=1\right\}\\
			&=\lim\limits_{\epsilon\to 0}\inf \left\{\mathscr{L}_{\epsilon}(H,\{F_n\}_{n=1}^{\infty},\Phi): \mu(H)=1\right\}.
		\end{aligned}
		$$
	\end{theorem}

	\bigskip \noindent{\bf Acknowledgement} The authors are grateful to the referee for a careful reading of the paper and a multitude of corrections and helpful suggestions.

	\bigskip \noindent{\bf Data availability}
	Data sharing not applicable to this article as no datasets were generated or
	analysed during the current study.

	\bigskip \noindent{\bf Declarations}
	
	\bigskip \noindent{\bf Conflict of interest}
	The authors have no financial or proprietary interests in any material discussed
	in this article.
	
	\bigskip \noindent{\bf Ethical Approval}
	Not applicable.

\end{document}